\documentclass[10pt]{article}
\usepackage{amsmath, amssymb, amsthm}
\usepackage{appendix}
\usepackage{captdef}
\usepackage{cite}

\usepackage{fullpage}

\usepackage{hyperref}

\makeindex
\usepackage{paralist}
\usepackage{psfrag}
\usepackage{graphicx}
\usepackage{graphics}
\usepackage{color}

\newcommand{\R}{\mathbb{R}}

\newtheorem{theo}{Theorem}[section]
\newtheorem{coro}[theo]{Corollary}
\newtheorem{lemma}[theo]{Lemma}
\newtheorem{prop}[theo]{Proposition}
\theoremstyle{definition}

\newtheorem{exam}[theo]{Example}

\newcommand{\cM}{{\mathcal M}}   
\newcommand{\cR}{{\mathcal R}}   

\renewcommand{\S}{{\mathbb S}^{N-1}}

\newcommand{\eps}{\varepsilon}      

\renewcommand{\phi}{\varphi}

\def\sideremark#1{\ifvmode\leavevmode\fi\vadjust{\vbox to0pt{\vss% the remark
 \hbox to 0pt{\hskip\hsize\hskip1em%                          will appear only
 \vbox{\hsize2.1cm\tiny\raggedright\pretolerance10000%          on the side
  \noindent #1\hfill}\hss}\vbox to15pt{\vfil}\vss}}}%

\numberwithin{equation}{section}

\title{
Symmetry properties of sign-changing solutions to nonlinear parabolic equations in unbounded domains
}
\author{
Juraj F\"{o}ldes\footnote{Department of Mathematics, University of Virginia
322 Kerchof Hall, Charlottesville, VA 22904-4137, foldes@virginia.edu.}, 
Alberto Salda\~{n}a\footnote{Instituto de Matem\'aticas, Universidad Aut\'onoma de M\'exico, Circuito Exterior, Ciudad Universitaria, 04510 Coyoac\'an, Ciudad de M\'exico, M\'exico, alberto.saldana@im.unam.mx.},\ \&
Tobias Weth\footnote{Institut f\"{u}r Mathematik, Johann Wolfgang Goethe-Universit\"{a}t Frankfurt, Robert-Mayer-Str. 10, D-60054 Frankfurt, weth@math.uni-frankfurt.de.} 
}
\date{}

\begin{document}
\maketitle

\begin{abstract}
We study the asymptotic (in time) behavior of positive and sign-changing solutions to nonlinear parabolic problems in the whole space or in the exterior of a ball with Dirichlet boundary conditions. We show that, under suitable regularity and stability assumptions, solutions are asymptotically (in time) foliated Schwarz symmetric, i.e., all elements in the associated omega-limit set are axially symmetric with respect to a common axis passing through the origin and are nonincreasing in the polar angle.  We also obtain symmetry results for solutions of Hénon-type problems, for equilibria (i.e. for solutions of the corresponding elliptic problem), and for time periodic solutions.
\end{abstract}


\section{Introduction}
In this paper we study the asymptotic behavior of (possibly sign-changing) classical uniformly bounded solutions of
\begin{equation}\label{model}
\begin{aligned}
 u_t - \Delta u &=  f(t,|x|,u) && \qquad (x,t) \in \Sigma \times (0,\infty),\\
    u(x,t) &= 0  && \qquad (x,t)\in \partial \Sigma \times (0,\infty),\\
   \ \ u(x,0) &= u_0(x) && \qquad x\in \Sigma,
\end{aligned}
\end{equation}
where $\Sigma$ is either the whole space $\mathbb R^N$ or the complement of a ball $\R^N\backslash B_R(0)$ in dimension $N\geq 2$ for some $R>0$.  The initial profile $u_0$ and the nonlinearity $f$ satisfy some regularity and stability assumptions detailed below.

A solution $u$ is said to be asymptotically symmetric if all the elements in the $\omega$-limit set $\omega(u)$, defined as 
\begin{equation}\label{omega}
 \omega(u):=\{z\in C_0(\overline{\Sigma}) :
 z(x)=\lim_{n\to\infty} u(x,t_n) \text{ for $x\in\Sigma$ and some } t_n\to\infty\},
\end{equation}
share some symmetry.  Here, $C_0(\overline{\Sigma})$ is the space of continuous functions which decay to zero at infinity and vanish on $\partial \Sigma$, equipped with the supremum norm.  
Standard parabolic estimates (see Lemma~\ref{regularity:lemma}) yield that
\begin{align}\label{omega:unif}
\lim_{t\to\infty}\operatorname{dist}_{C_0(\overline{\Sigma}
)}(u(\cdot,t),\omega(u))=0\,,
\end{align}
and therefore the asymptotic symmetry implies that the solution is more and more symmetric as $t \to \infty$.  

We are interested in a particular kind of symmetry sometimes referred to as foliated Schwarz symmetry.  We say that a function $u\in C(\Sigma)$ is \textit{foliated Schwarz symmetric with respect to some unit vector $p\in \S =\{x\in\mathbb R^N: |x|=1\}$}, if $u$ is axially symmetric with respect to the axis $\mathbb R p$ and nonincreasing in the polar angle $\theta:= \operatorname{arccos}(\frac{x}{|x|}\cdot p)\in [0,\pi].$ If $u$ is strictly decreasing in $\theta$, then we say that $u$ is \emph{strictly foliated Schwarz symmetric}. 

From the asymptotic symmetry point of view, the results we present in this paper are, as far as we know, the first to consider sign-changing solutions in unbounded domains, non-monotone spatial dependences on the nonlinearity, and unbounded domains different from $\R^N$ (we give an account of previously known results below).   In this more general setting, however, we need to impose a geometric assumption on the initial profile $u_0$ to guarantee that all functions in the $\omega$-limit set are foliated Schwarz symmetric. 

Before we state our theorems in full generality, we illustrate our results with two paradigmatic particular cases.  First, we exploit the (non-monotone) spatial dependence of the coefficients of a nonautonomous Hénon-type problem to guarantee a symmetrizing effect as $t\to\infty.$ 

\begin{theo}\label{thm1}
 Let $\Sigma$ be either $\R^N$ or $\R^N\backslash B_1(0)$, $a,b\in C([0,\infty))\cap L^\infty([0,\infty))$, and let $\eta>0$ be such that
\begin{align}\label{eta:intro}
\inf\limits_{t>0}b\geq \eta. 
\end{align}
 Let $p>1$, $0\leq \alpha<\beta$, and let $u\in C^{2,1}(\Sigma\times(0,\infty))\cap C(\overline{\Sigma}\times[0,\infty))\cap L^\infty(\Sigma\times[0,\infty))$ be a solution of
\begin{equation}\label{p1}
\begin{aligned}
 u_t - \Delta u &= a(t)|x|^\alpha\, |u|^{p-1}u-b(t)|x|^\beta\, u && \qquad    (x, t) \in \Sigma\times (0,\infty),\\
 u &=0 && \qquad (x, t) \in \partial \Sigma\times(0,\infty),\\
 u(x,0)&=u_0(x)&& \qquad  x\in \Sigma,
\end{aligned}
\end{equation}
where $u_0\in L^\infty(\Sigma)\cap C(\overline{\Sigma})$ satisfies that
 \begin{equation}\label{U0intro}
 \begin{aligned} 
u_0(x_1,x_2,\ldots,x_N)&\geq u_0(-x_1,x_2,\ldots,x_N)\quad \text{ for all $x\in \Sigma$ with $x_1>0$,}\\
u_0(x_1,x_2,\ldots,x_N)&>u_0(-x_1,x_2,\ldots,x_N)\quad \text{ for some $x\in \Sigma$ with $x_1>0$.}
\end{aligned}  
 \end{equation}
Then,  $u$ is asymptotically foliated Schwarz symmetric, that is, there is $p\in \S$ such that $z$ is foliated Schwarz symmetric with respect to $p$ for all $z\in\omega(u)$.
\end{theo}

Observe that the direction of the symmetry axis $p$ is not determined by the equation, which is invariant under rotations, $p$ is determined by $u_0$. Theorem \ref{thm1} is a consequence of the more general Theorem \ref{strong:thm} below.

\medskip

Our results also cover translationally invariant problems, where we can show that (possibly sign-changing) solutions are asymptotically signed and radially symmetric with respect to \emph{some} center. Unlike for Hénon-type problems, a center of symmetry is not fixed a priori if the equation is translationally invariant; and therefore a suitable assumption on the initial condition $u_0$ plays a key role to guarantee the asymptotic symmetry. In the next section we present more general conditions on $u_0$; however, to simplify the presentation and to illustrate the main ideas, in the next theorem we assume that the supports of the positive and the negative parts of $u_0$ are strictly separated by the cones $K^+:=\{x=(x_1,x')\in\R^N\::\: x_1>1\text{ and }|x'|<|x_1-1|\}$ and $K^-:=\{x=(x_1,x')\in\R^N\::\: x_1<-1\text{ and }|x'|<|x_1+1|\}$, see Figure \ref{fig11} below. 

\begin{theo}\label{thm2}
 Let $p>1$ and $a,b\in C([0,\infty))\cap L^\infty([0,\infty))$ such that \eqref{eta:intro} holds. Let 
 $u\in C^{2,1}(\R^N\times(0,\infty))\cap C(\R^N\times[0,\infty))\cap L^\infty(\R^N\times[0,\infty))$ be a solution of
\begin{equation}\label{p2}
\begin{aligned}
 u_t - \Delta u + b(t)\, u &= a(t)\, |u|^{p-1}u \qquad && (x, t) \in \R^N\times (0,\infty),\\
  u(x,0)&=u_0(x) \qquad && x\in \R^N,
\end{aligned}
\end{equation}
where $u_0\in C_0(\R^N)\backslash\{0\}$ is such that
\begin{align}\label{u0h}
\{x\in\R^N\::\: u_0(x)>0\}\subset K^+,\quad\{x\in\R^N\::\: u_0(x)<0\}\subset K^- \,.
\end{align}
Assume furthermore that $u$ has uniform decay at spatial infinity
\begin{align}\label{udintro}
\lim\limits_{|x|\to\infty} \sup\limits_{t>0} u(x,t)=0. 
\end{align}
If $0\not\in\omega(u)$, then $u$ is asymptotically signed and radially symmetric, that is, there is $q\in \R^N\backslash\{x_1=0\}$ such that all elements in $\omega(u)$ are radially symmetric with respect to $q$ and
\begin{enumerate}
 \item $z>0$ in $\R^N$ for all $z\in\omega(u)$, if $q\in\{x_1>0\}$,
 \item $z<0$ in $\R^N$ for all $z\in\omega(u)$, if $q\in\{x_1<0\}$.
\end{enumerate}

\end{theo}

Note that, in this setting, there is an additional complication if $0$ belongs 
to the $\omega$-limit set.  If $0\in\omega(u)$, then for a sequence of times $t_n\to\infty$ the solution tends uniformly to zero, and therefore the (symmetrizing) influence from the initial condition weakens and the asymptotic symmetry of solutions becomes unclear. This is why Theorem \ref{thm2} poses an alternative: either $0\in\omega(u)$ or the solution $u$ is asymptotically signed and radially symmetric.  Theorem \ref{thm2} is a consequence of the more general Theorem \ref{theorem:unbounded} below together with an argument involving the translational invariance of the equation. 
\begin{figure}[!h]\label{fig11}
\begin{center}
 \includegraphics[width=.50\textwidth]{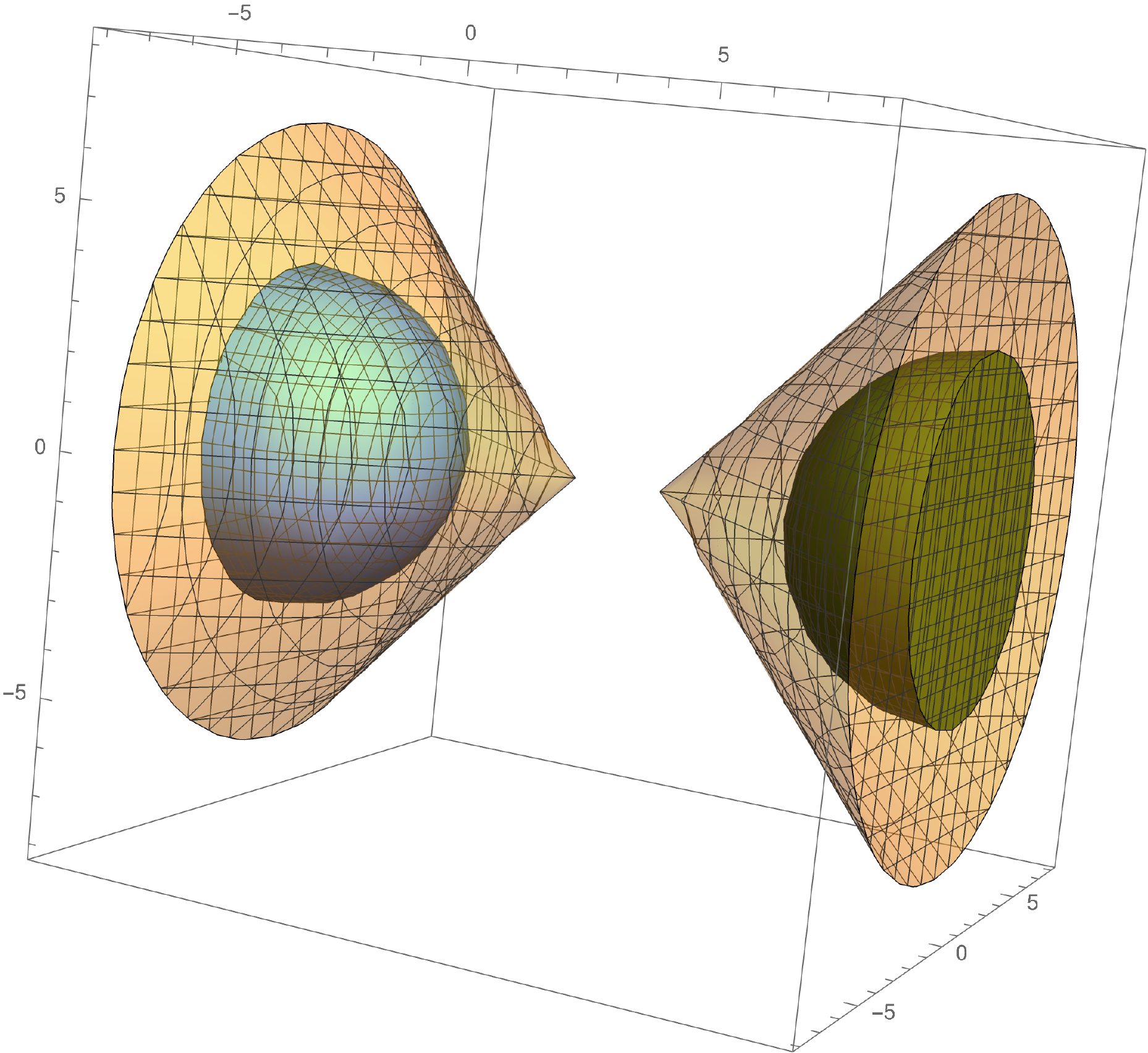}
 \end{center}
\caption{Geometry of the condition \eqref{u0h} in $\R^3$, where the support of the positive part of $u_0$ (represented by the dark region on the right) lies inside a cone in the halfspace $\{x_1>1\}$ and the support of the negative part of $u_0^-$ (represented by the dark region on the left) is inside a cone in the halfspace $\{x_1<-1\}$.}\label{hfig}
\end{figure}

\subsection*{Main results}

In the following, we give a more abstract framework to state general results from which Theorems \ref{thm1} and \ref{thm2} can be deduced. Our results are complementary to the papers \cite{polacik,polacik:unbounded}, where asymptotic radial symmetry of positive solutions is proved in bounded and unbounded domains.  More specifically, the conclusions in \cite{polacik, polacik:unbounded} are stronger (radial versus foliated Schwarz symmetry) and there are no geometric assumptions on the initial condition; however, 
\cite{polacik, polacik:unbounded} require positivity of solutions and monotone dependence of $f$ on $|x|$ (we remark that the theorems in \cite{polacik:unbounded} do not consider spatial dependences on $f$, but the proof can be adjusted to include nonincreasing radial spatial dependences).  The main tool in \cite{polacik, polacik:unbounded} is the \emph{parabolic moving plane method} (MPM), which is a perturbation technique strongly based on maximum principles, Harnack inequalities, parabolic regularity, and the construction of suitable subsolutions.  The lack of compactness in unbounded domains is the main difficulty when trying to characterize the symmetry of solutions of~\eqref{model}, and in fact, symmetry-breaking phenomena in unbounded domains are known, see for example \cite{PY04}. For other results regarding reflectional symmetry and monotonicity in bounded domains via MPM, see  \cite{f1,f2,f3,f4,f5,f6,fp}.

In this paper we consider sign-changing solutions, spatially non-monotone nonlinearities, and exterior domains.  In this setting, a standard MPM cannot be used; instead, we use a variant sometimes called \emph{rotating plane method} (RPM).  We explain the differences between these two techniques below.  For similar results in bounded domains for parabolic equations and systems, we refer to \cite{saldana:2016,saldana:2015,saldana-weth,torino}, and for elliptic problems to \cite{GPW10,wethsurvey}.

We introduce some notation to formulate our results. Recall that $\Sigma$ can be the whole space without a ball of radius $R>0$, $\Sigma=\R^N\backslash B_R(0)$, or the whole space $\Sigma=\R^N$, in which case we set $R=0$. For a vector $e\in \S$, we consider the hyperplane $H(e):=\{x\in \mathbb R^N: x\cdot e=0\}$ and the half domain $\Sigma(e):=\{x\in \Sigma: x\cdot e>0\}.$  We also write $\sigma_e:\Sigma\to \Sigma$ to denote reflection with respect to $H(e),$ i.e. $\sigma_e(x):=x-2(x\cdot e)e$ for each $x\in \Sigma.$

We consider a classical global solution $u\in C^{2,1}(\Sigma\times(0,\infty))\bigcap C(\overline{\Sigma}\times[0,\infty))$ of~\eqref{model} satisfying the following assumptions.
\begin{itemize}
 \item[$(U_0)$] \emph{(Initial reflectional inequality)} The initial profile $u_0$ belongs to $C_0(\overline{\Sigma})$ and there exists $e\in \S$ such that 
 \begin{align*}
\text{ $u_0(x)\geq u_0(x^e)$ for all $x\in \Sigma(e)$ and $u_0\not\equiv u_0\circ \sigma_e.$  }
 \end{align*}
\item [$(U_1)$] \emph{(Uniform decay)} $\lim\limits_{|x|\to\infty} \sup\limits_{t>0} u(x,t)=0.$
\item [$(U_2)$] \emph{(Uniform boundedness)} $||u||_{L^\infty(\Sigma\times(0,\infty))}=:M_1<\infty$.
\end{itemize}

 Define $I:=[R,\infty)=\overline{\{|x|:x\in\Sigma\}}.$  We now state our first assumption on $f$.
\begin{itemize}
 \item [$(f_0)$] \emph{(Boundedness at zero)} For every $r>0$ there is $K_r>0$ such that
 \begin{align*}
\sup\limits_{s\in [R,r], t>0}|f(t,s,0)|<K_r.
 \end{align*}
\end{itemize}
Assumption $(f_0)$ is used to obtain the equicontinuity of $u$, see Lemma~\ref{regularity:lemma}.

Next, we impose some standard regularity on $f$ and a crucial sign assumption on $f_u$ outside a compact set. This last assumption has two important variants (see $(f_2)'$ and $(f_2)$ below), and we divide our results accordingly. 

\subsubsection*{Strong stability outside compact sets}

Let $\lambda_1$ denote the principal eigenvalue of the Dirichlet Laplacian on the unit ball $B_1(0)\subset \R^N$. Assume that
\begin{itemize}
  \item [$(f_1)'$] \emph{(Regularity)} the nonlinearity $f: [0,\infty) \times I \times \R \to \R$, $(t,r,u) \mapsto f(t,r,u)$ is continuous in $t,r,$ and locally uniformly Lipschitz in $u,$ that is, for every bounded interval $J\subset I$ and $K>0$ there is $C=C(K,J)>0$ such that
\begin{align*}
\sup_{r\in J,t>0,u,v\in[-K,K]}|f(t,r,u)-f(t,r,v)| \leq C|u-v|.
\end{align*}
 In particular, $f_u(t,r,u)$ exists for almost every $u$ and it is locally bounded.
 \item [$(f_2)'$] \emph{(Strong stability)}  given $M>0$ there are $\rho>R$, $\eps>0$, and an interval $J\subset (R,\rho)$ such that
\begin{align}\label{Idef}
 &\max_{r>\rho,u\in[-\eps,\eps]}f_u(t,r,u)<-\max_{
 \substack{
 {r\in J,}
 \\
 {u\in[-M,M]}
 }}
 |f_u(t,r,u)|- 4\frac{\lambda_1}{|J|^2}\qquad \text{for all $t>0$}.
\end{align}
\end{itemize}

If the bound $M_1$ from $(U_2)$ is known, then it suffices to suppose that $(f_2)'$ is satisfied for $M=M_1$. Observe that the assumption $(f_2)'$ can only be satisfied by nonlinearities with a spatial dependence.  This is one of the novelties in Theorem \ref{strong:thm} below: we show that a condition on the spatial dependence of the nonlinearity (given by \eqref{Idef}) yields a symmetrizing effect on the solution for large times.

\medskip

Our main result in this setting is the following.

\begin{theo}\label{strong:thm}
 Let $u\in C^{2,1}(\Sigma\times(0,\infty))\bigcap C(\overline{\Sigma}\times[0,\infty))$ be a classical global solution of~\eqref{model} satisfying $(U_0)$-$(U_2)$ and assume that $(f_0)$, $(f_1)'$, and $(f_2)'$ hold.  Then $u$ is asymptotically foliated Schwarz symmetric, that is, there is $p\in \S$ such that $z$ is foliated Schwarz symmetric with respect to $p$ for all $z\in\omega(u)$.
\end{theo}

Note that the axis of symmetry $\R p$ is fixed for all elements in $\omega(u)$.  A typical situation where $(f_2)'$ is satisfied is the case of a bounded nonlinearity plus a suitable potential. The next corollary shows an example.

\begin{coro}\label{coro:sim}
Let $f$ satisfy $(f_1)'$ and be such that
\begin{align*}
&\sup_{r\geq R,t>0,s\in \R} |f_u(t,r,s)|=:C_0 <\infty.
\end{align*}
Let $u\in C^{2,1}(\Sigma\times(0,\infty))\bigcap C(\Sigma\times[0,\infty))$ be a classical solution satisfying $(U_0)$--$(U_2)$ of
\begin{equation}\label{particularcase}
\begin{aligned}
 u_t-\Delta u+V(|x|)u&= f(t,|x|,u) &&\text{ in }\Sigma \times (0,\infty),\\
       \ \ u(x,0) &= 0 &&\text{ on }\partial\Sigma\times(0,\infty),\\
       \ \ u(x,0) &= u_0(x) &&\text{ for }x\in \Sigma,
\end{aligned}
\end{equation}
where $V:[0,\infty)\to \R$ is a continuous function such that 
\begin{align}\label{V}
&\min_{(\rho,\infty)}V>\max_{{\mathcal J}}|V| + 4\frac{\lambda_1}{|{\mathcal J}|^2}+2C_{0},
\end{align}
for some $\rho>R$ and some interval ${\mathcal J}\subset [R,\rho)$. Then $u$ is asymptotically foliated Schwarz symmetric.
\end{coro}
Observe that~\eqref{V} is always satisfied if $V(r)\to\infty$ as $r\to\infty$, and there is no restriction on the behavior of $V$ in $[0,\rho)\backslash I$. We also remark that $0<\lambda_1\leq \frac{N}{2}\left( \sqrt{\frac{N}{2} +1}  + 1 \right)^2,$ (see \cite{ch82}) where $N$ is the dimension.

The proof of Theorem~\ref{strong:thm} is based on an RPM which extends the ideas from the bounded domain case \cite{saldana-weth,polacik} to unbounded domains. Similarly as the MPM, the RPM can be divided in three parts: the start, the perturbation step, and the characterization of symmetry. The start relies strongly on assumption $(U_0)$, maximum principles, and a suitable linearization of~\eqref{model} (see Lemma~\ref{T:l}).  The perturbation step is based on maximum principles (Lemma~\ref{weakmp}), Harnack inequalities (Lemma~\ref{harnack:inequality}), parabolic regularity (Lemma~\ref{regularity:lemma}), and the construction of suitable subsolutions (Lemmas~\ref{subsolb} and~\ref{lemma5}). A crucial aspect in this part is the fine use of constants which have precise dependences in order to construct and control a suitable perturbation. This is particularly delicate when considering unbounded domains and nonlinearities with a spatial dependence.  Here hypothesis $(U_1)$, $(U_2)$, and $(f_2)'$ allow to control the solution outside large compact sets, and this is essential to compensate for the lack of compactness in unbounded domains. The last part, the characterization of symmetry, relies on a geometric characterization of foliated Schwarz symmetry in terms of reflectional inequalities (Lemma~\ref{sec:char-foli-schw-1}), such characterizations were first presented in \cite{brock} in the study of elliptic variational problems.

\subsubsection*{Weak stability outside compact sets}

Next we use a set of assumptions which is closer to that of \cite{polacik:unbounded}. Assume that
\begin{itemize}
\item [$(f_1)$] (Regularity) $f: [0,\infty) \times I \times \R \to \R$, $(t,r,u) \mapsto f(t,r,u)$ is continuous in $t,r,$ continuously differentiable in $u$, and satisfies that
\begin{align*}
&\lim_{u\to v}\sup_{r\in I,t>0}|f_u(t,r,u)-f_u(t,r,v)|=0 &&\text{for every $v\in\mathbb R,$}\\
&\sup_{r\in J,t>0,s\in[-K,K]}|f_u(t,r,s)|<\infty &&\text{ for every $K>0$ and $J\subset\subset I$}.
\end{align*}
\item [$(f_2)$] \emph{(Stability)} there are a constants $\rho,\gamma,\varepsilon>0$ such that 
\begin{align*}
 f_u(t,|x|,s)<-\gamma\qquad \text{for all $|x|>\rho$, $t\geq 0$, and $s\in(-\varepsilon,\varepsilon)$.}
\end{align*}
\end{itemize}

Here $J\subset\subset I$ means that $J$ is \emph{compactly contained} in $I$, that is, $\overline{J}$ is compact and a subset of $I$.  The difference between $(f_2)'$ and $(f_2)$ is that $\gamma>0$ in $(f_2)$ can be arbitrarily small.  This allows to consider, for example, a problem as in~\eqref{particularcase}, where $V$ is possibly unbounded from below but $V(r_n)=-\gamma$ for a sequence $r_n\to\infty$ and some $\gamma\in(0,1)$ small.  This weaker assumption $(f_2)$ implies, however, a weaker control on the solution at spatial infinity and forces the RPM to use a different subsolution given by Lemma~\ref{lemma5}, which in turn requires some knowledge on the elements in $\omega(u)$. As a consequence, our main result under the weaker assumption $(f_2)$ describes an alternative. 

\begin{theo}\label{unbounded}
Assume $(f_0)$, $(f_1)$, $(f_2)$, and let $u\in C^{2,1}(\Sigma\times(0,\infty))\bigcap C(\overline{\Sigma}\times[0,\infty))$ be a classical global solution of~\eqref{model} satisfying $(U_0)$, $(U_1)$, and $(U_2)$.  Then one of the following alternatives happen:
\begin{enumerate}
 \item $u$ is asymptotically foliated Schwarz symmetric with respect to some $p\in \S$, that is,  all elements in $\omega(u)$ are foliated Schwarz symmetric with respect to $p$.  Moreover, all elements in $\omega(u)$ are strictly decreasing in the polar angle.
 \item There exists $z\in\omega(u)$ such that $z\equiv z\circ \sigma_e$ with $e$ is in $(U_0)$.
\end{enumerate}
\end{theo}

The proof of Theorem~\ref{unbounded} follows an RPM and some ideas from \cite{polacik:unbounded}. We emphasize that, in general, the second alternative in Theorem~\ref{unbounded} can occur; for example, if $u_0\geq 0,$ then \cite[Theorem 1.1]{polacik:unbounded} implies that all the elements in $\omega(u)$ are radially symmetric with respect to some center.  Note that the first alternative rules out radial symmetry (with respect to the origin) because of the monotonicity properties in the polar angle.  In Section~\ref{e1}, we show an example of a solution $u$ for which $\omega(u)$ only has nodal strictly foliated Schwarz functions. This example also shows that the set of solutions satisfying the assumptions of Theorem~\ref{unbounded} is nonempty.  Furthermore, we show in Section~\ref{e2} that the symmetry in the second alternative in Theorem~\ref{unbounded} may not propagate to the other elements in $\omega(u)$; more precisely, we show in Section~\ref{e2} two examples (one in a bounded domain and one in an unbounded domain) of problems whose solution $u$ has both radially symmetric and strictly foliated Schwarz functions in $\omega(u)$; in particular, this shows that the presence of a radially symmetric function in $\omega(u)$ does not imply, in general, that all
 elements in $\omega(u)$ are radially symmetric.

Whenever the second alternative happens, this creates an obstacle to start the method, since one cannot use the subsolution in Lemma~\ref{lemma5}.  This kind of obstacles do not appear in a MPM framework (as in \cite{polacik:unbounded}), because the starting step in the method relies on the positivity of solutions, the uniform decay at spatial infinity, and on the stability of 0 given by assumption $(f_2)$.

The following are direct corollaries for elliptic and periodic parabolic equations, where the second alternative can be discarded using the maximum principle.

\begin{coro}\label{corollary:periodic}
Assume $(f_0)$--$(f_2)$,  and let  $u\in C^{2,1}(\Sigma\times(0,\infty))\bigcap C(\overline{\Sigma}\times[0,\infty))$ be a classical bounded periodic solution of~\eqref{model} that satisfies $(U_0)$ and that
$\lim\limits_{|x|\to\infty}u(x,t)=0$ for $t\geq 0$. Then there exists $p\in \S$ such that, for all $t\geq 0,$ $u(\cdot,t)$ is foliated Schwarz symmetric with respect to $p$ and $u$ is strictly decreasing in the polar angle.
\end{coro}

\begin{coro}\label{corollary:elliptic}
 Let $f:I \times \R \to\R$, $(r,u) \mapsto f(r,u)$ be continuous in $r$ and continuously differentiable in $u,$ with $f_u$ uniformly continuous and bounded with respect to $r.$   
 Moreover, assume that there are constants $\rho,\gamma>0$ such that $f_u(|x|,0)<-\gamma$ for all $|x|>\rho.$ Let $u\in C^{2}(\Sigma)\bigcap C(\overline{\Sigma})$ be a classical uniformly bounded solution of 
\begin{equation*}
\begin{aligned}
 -\Delta u =  f(|x|,u)\quad \text{ in }\Sigma,\qquad
    u(x) = 0 \quad \text{ on }\partial \Sigma,\qquad
    \lim_{|x|\to\infty}u(x)=0.
\end{aligned}
\end{equation*} 
If there is $e\in \S$ such that $u(x)\geq u(\sigma_e(x))$ in $\Sigma(e)$ and $u\not\equiv u\circ\sigma_e,$ then $u$ is foliated Schwarz symmetric with respect to some $p\in \S$ and $u$ is strictly decreasing in the polar angle.
\end{coro}

Similar corollaries can be stated for Theorem~\ref{strong:thm}. In fact, for elliptic and periodic parabolic problems a slightly simpler proof can be done, where assumption $(f_1)$ can be weakened (as in \cite{saldana-weth}, since Lemma~\ref{lemma5} is not needed).  

Under a more \textquotedblleft geometrically stable\textquotedblright\ assumption on $u_0$, we can prove that if some $z\in\omega(u)$ has $H(e)$ as symmetry hyperplane, then $z$ is radial. This assumption is the following:
\begin{itemize}
 \item [$(U_0)'$] The initial profile $u_0$ belongs to $C_0(\overline{\Sigma})$ and there exists an open set $U\subset \S$ such that 
 \begin{align*}
u_0\geq u_0\circ\sigma_e\quad \text{ in $\Sigma(e)$ for all $e\in U.$}
 \end{align*}
\end{itemize}

\begin{theo}\label{theorem:unbounded}
Assume the same hypothesis as in Theorem~\ref{unbounded} but with $(U_0)'$ instead of $(U_0)$.  Then, only one of the following alternatives happen:
\begin{enumerate}
 \item $u$ is asymptotically foliated Schwarz symmetric with respect to some $p\in \S$, i.e. all elements in $\omega(u)$ are strictly foliated Schwarz symmetric with respect to $p$; that is, all the elements in $\omega(u)$ are axially symmetric and strictly decreasing in the polar angle.
 \item There exists at least one $z\in\omega(u)$ such that $z$ is radially symmetric with respect to the origin.
\end{enumerate}
\end{theo}
Functions that satisfy $(U_0)'$ are, for example, positive functions with support in a one-sided cone, or sign changing functions with support in two (disjoint) cones being positive in one side and negative in the other (see Figure~\ref{u0} and Section~\ref{e1}).  The proof of Theorem~\ref{theorem:unbounded} uses Theorem~\ref{unbounded} plus some geometric results that characterize the rigidity that reflectional inequalities impose on a function (see Lemmas~\ref{lemma2:cor} and~\ref{lemma2.5}).

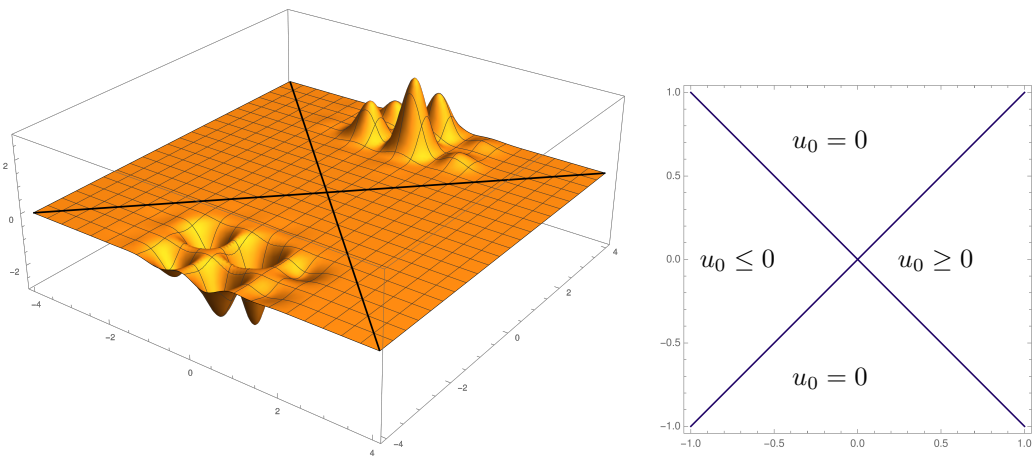
\begin{figure}[!h]
\begin{center}
\includegraphics[width=.50\textwidth]{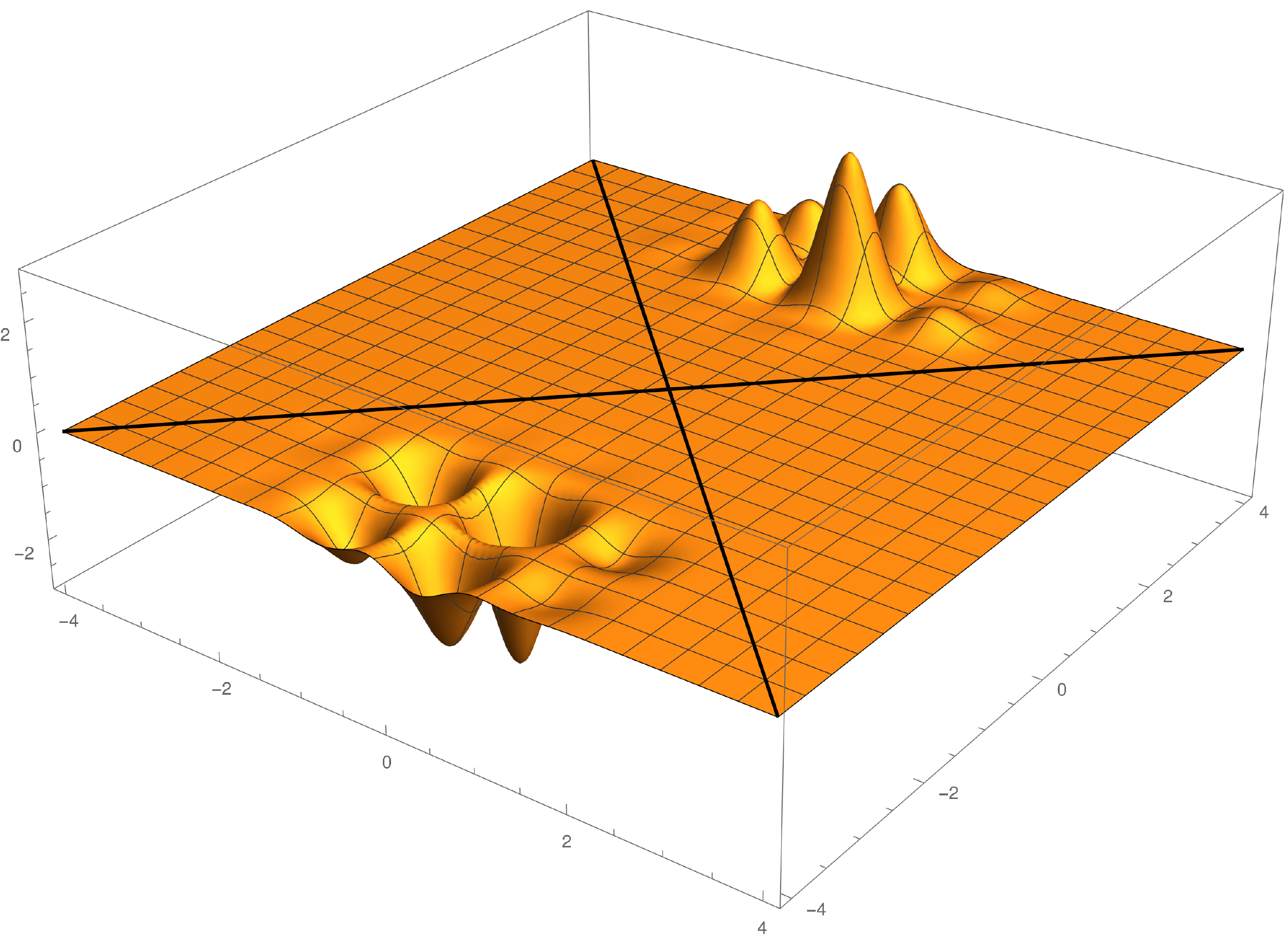}
\begin{picture}(150,140)
    \put(0,5){\quad\includegraphics[width=.30\textwidth]{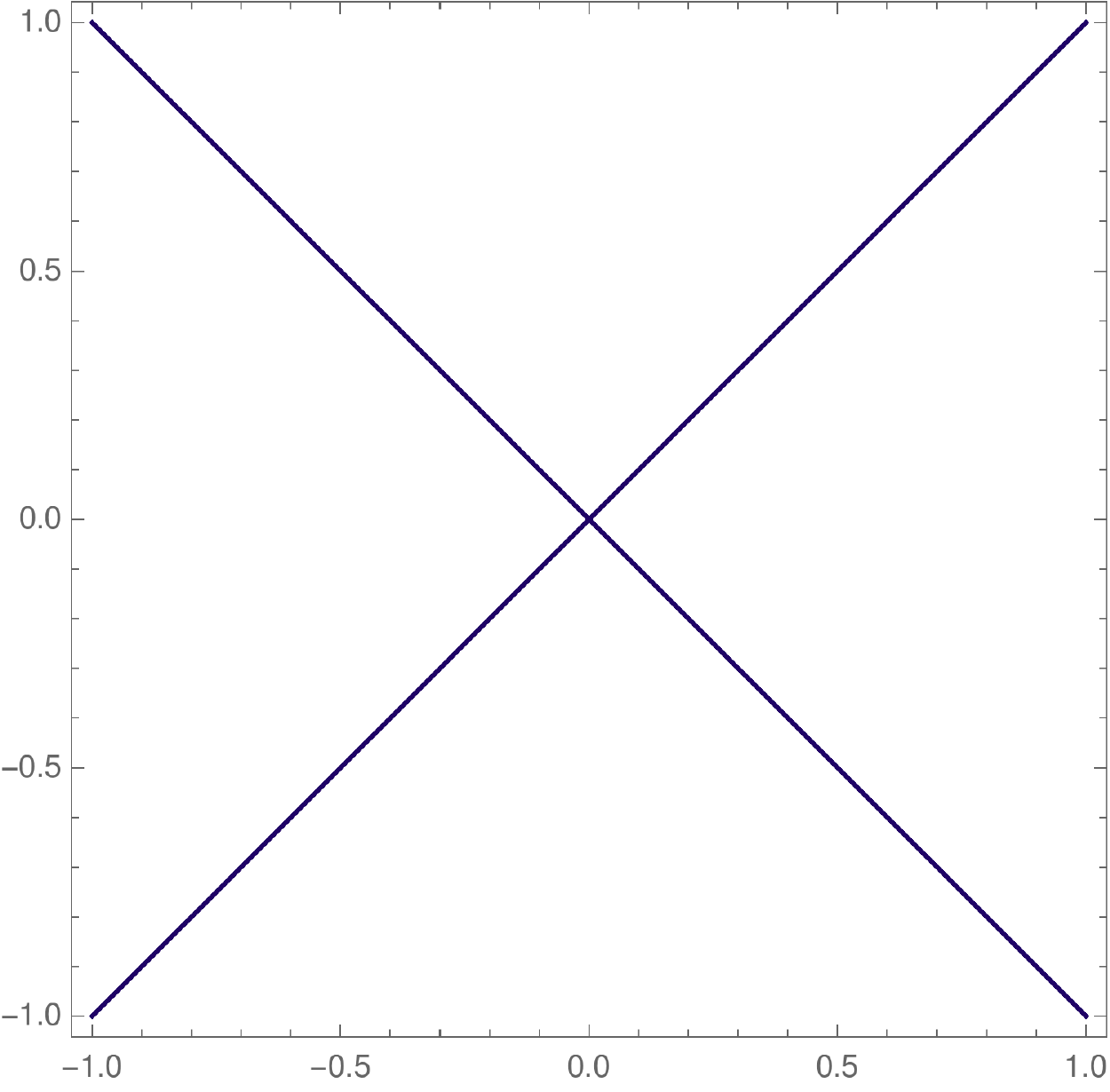}}
    \put(60,120){$u_0=0$}
    \put(60,30){$u_0=0$}
    \put(25,75){$u_0\leq 0$}
    \put(100,75){$u_0\geq 0$}
  \end{picture}
  \end{center}
\caption{Example of a function $u_0$ satisfying $(U_0)'$.}\label{u0}
\end{figure}

Finally, we remark that Theorem~\ref{unbounded} can be extended to the quasilinear setting as in \cite{polacik:unbounded}. The main obstacle for a generalization to fully nonlinear problems is, as noted in \cite[Section 4]{polacik:unbounded}, the generalization of the subsolution given in Lemma~\ref{lemma5}. However, Theorem~\ref{strong:thm} can be extended to fully nonlinear problems as in the bounded domain case \cite{saldana-weth,polacik}.  Our results consider only the semilinear setting for simplicity and to make the main ideas more transparent. 

\subsubsection*{On the uniform decay assumption}

In Theorem \ref{thm1}, the assumption $(U_1)$ is not stated as a hypothesis, since it can be obtained using the particular form of the equation and comparison principles with suitable supersolutions.  In this part of the introduction, we discuss some other conditions that can be used to guarantee that $(U_1)$ holds.  

In \cite[Corollary 1.2]{busca} it is proved that $(U_1)$ is satisfied if the initial condition has compact support and the nonlinearity does not depend on $|x|.$  Following the same idea, we give below sufficient conditions to achieve the uniform decay assumption $(U_1)$ in settings where nonradial solutions (and nonradial limit profiles) could be expected, for example, if the nonlinearity $f:[0,\infty)\times[0,\infty)\times\mathbb R\to\mathbb R$ is such that, for some $R_0>0$,
\begin{align}\label{weak:monotonicity}
f(t,r_1,u)\leq f(t,r_2,u),\ \text{ for all } 0\leq r_1 \leq R_0 \leq r_2 \text{ or } R_0\leq r_1 \leq r_2,
\end{align}
and all $t>0$, $u\in\mathbb R.$ In particular, $f$ can be non-monotone inside a ball of radius $R_0.$ Then, we have the following lemmas.

\begin{lemma}\label{UD1} Let $\Sigma=\mathbb R^N$ and $u$ a solution of~\eqref{model}, assume that $\partial_u f(t,r,u)$ is bounded uniformly in $t$ and $r,$ and that $f(t,r,u)$ satisfies~\eqref{weak:monotonicity}.  If 
\begin{enumerate}
 \item $u_0\geq 0,$
 \item $u_0$ has compact support,
 \item $||u(\cdot,t)||_{L^\infty(\mathbb R^N)}\leq C$ and $||u(\cdot,t)||_{L^q(\mathbb R^N)}\leq C$ for some $C>0,$ $1\leq q<\infty$ and all $t>0,$
\end{enumerate}
then $u$ satisfies $(U_1)$.
\end{lemma}

\begin{lemma}\label{UD2}Let $\Sigma=\mathbb R^N$ and assume $\partial_u f(t,r,u)$ is bounded uniformly in $t$ and $r,$ and that $f(t,r,u)$ satisfies~\eqref{weak:monotonicity}.  If 
\begin{enumerate}
\item $u_0$ has compact support,
\item the problems
\begin{equation}\label{dfv}
v_t -\Delta v = f(t,|x|,v)\quad \text{ in }\mathbb R^N \times (0,\infty),\qquad
   v(x,0) = u^\pm_0(x)\quad \text{ for } x\in \R^N,
\end{equation}
where $u^+_0$ and $u^-_0$ denote the positive and negative part of $u_0$ respectively, have global solutions $v_1$ and $v_2$ such that $||v_i(\cdot,t)||_{L^\infty(\mathbb R^N)}\leq C$ and $||v_i(\cdot,t)||_{L^q(\mathbb R^N)}\leq C$ for some $C>0,$ $1\leq q<\infty$, and all $t>0,$ $i=1,2,$
\end{enumerate}
then any solution $u$ of~\eqref{model} satisfies $(U_1)$.
\end{lemma}

\subsubsection*{Organization of the paper}

The paper is organized as follows. In Section~\ref{auxiliary:lemmas} we collect several useful results regarding a linearization of~\eqref{model}, regularity, Harnack inequalities, maximum principles, the existence of suitable explicit subsolutions, and some geometric lemmas which are used in the characterization of symmetry. Section~\ref{symmetry:results} is devoted to the proofs of our main symmetry results including Theorems \ref{thm1} and \ref{thm2}. The proofs of Lemmas~\ref{UD1} and~\ref{UD2} can be found in Section~\ref{hypothesis:discussion}.  Finally, in Section~\ref{ex:sec}, we include a series of examples to show some of the different possible behaviors of the elements in the $\omega$-limit set; in particular, we exhibit an $\omega$-limit set with only strictly foliated Schwarz symmetric functions and an $\omega$-limit set with a strictly foliated Schwarz symmetric function and a radially symmetric element.

\section{Auxiliary Lemmas}\label{auxiliary:lemmas}

\subsection{Linearization}\label{linearization}
Let $R \geq 0$ be such that
\begin{align}\label{R}
\Sigma = \R^N\backslash B_R, 
\end{align}
where 
 $B_R=B_R(0) = \{x \in \R^N : |x| <  R\}$ is  the open ball of radius $R$ centered at the origin if $R>0$ and $B_0 := \emptyset$, that is, $\R^N\backslash B_0=\R^N$. 
Using the notation given in the introduction, define $\Sigma_1:=\Sigma(e_1)$, where $e_1=(1,0,\ldots,0)$. 

Throughout this section we assume that $f$ satisfies $(f_1)'$, which follows from $(f_1)$, and $(f_2)$ which follows from $(f_2)'$. 
Let $u$ be a classical solution of~\eqref{model} satisfying $(U_1)$ and $(U_2)$. 

\medskip

Fix $e \in \S$ and let $\Gamma_e:\Sigma_1 \to\Sigma(e)$ be a rotation that maps $e_1$ to $e$. Then, 
\begin{align*}
w_e(x,t):=u(\Gamma_e(x),t)-u(\sigma_e(\Gamma_e(x)),t)\qquad \text{ for $(x,t)\in\overline{\Sigma}_1\times[0,\infty),$ }
\end{align*}
is a classical solution of 
\begin{equation}\label{Pie}
 \begin{aligned}
 \partial_t w_e - \Delta w_e - c^e(x,t)w_e &= 0 \hspace{4.5cm} \text{ in }\Sigma_1\times (0,\infty),\\
 w_e(x,t) &= 0 \hspace{4.5cm} \text{ on } \partial \Sigma_1\times (0,\infty),\\
 w_e(x,0) &= u_0(\Gamma_e(x))-u_0(\sigma_e(\Gamma_e(x))) \qquad \text{ for }x\in \Sigma_1,
   \end{aligned}
\end{equation}
satisfying $(U_1)$ and $(U_2)$ with $u$ replaced by $w_e$,
where 
\begin{align*}
c^e(x,t)&:=\int_0^1 \partial_u f(t,|x|,su(\Gamma_e(x),t)+(1-s)u(\sigma_e(\Gamma_e(x)),t))ds
\end{align*}
is well defined by $(U_2)$ and  $(f_1)'$. Moreover,  since $u$ is uniformly bounded in time and the Lipschitz constant of $f$ is time independent,  for every bounded subset $U\subset \Sigma_1$, there is $\beta_U>0$ (independent of $e$) such that
\begin{align}\label{beta0}
 \sup_{(x,t)\in U\times[0,\infty)} |c^e(x,t)| < \beta_U \,.
\end{align}

Moreover, by $(U_1)$, $(f_1)$, and $(f_2)$ (or $(U_1)$, $(f_1)'$, and $(f_2)'$), there exists $\rho_1\geq \rho$ (see $(f_2)$ or $(f_2)'$ for the definition of $\rho$) such that
\begin{align}\label{rho1}
 \sup_{(x,t)\in (\Sigma_1\setminus B_{\rho_1}(0)) \times[0,\infty)}c^e(x,t)<-\gamma,
\end{align}
where $\gamma$ is given by $(f_2)$ or
\begin{align}\label{gammaf2}
 \gamma:=\max_{
 \substack{
 {r\in J,}
 \\
 {u\in[-M_1,M_1]}
 }}
 |f_u(t,r,u)|+4\frac{\lambda_1}{|J|^2}\qquad \text{if $(f_2)'$ is assumed},
\end{align}
with $M_1$ as in $(U_2)$ and $J$ as in $(f_2)'$.  Observe that $\gamma$ and $\rho_1$ are independent of $e \in \S$.

For the proofs of our main results we need $c^e$ to be negative not only far from the origin, but also near $\partial \Sigma_1$.  This is achieved by a modification of $w_e$ using appropriate subsolutions, as in \cite{bere, polacik,polacik:unbounded}; however, since we only consider radial domains, the proof is simpler and more explicit. We use $\chi_J$ to denote the  characteristic function of an interval $J\subset \R$.

\begin{lemma}\label{T:l}
 For every $e\in \S$, let $w_e$ be a classical solution of~\eqref{Pie} with $c^e$ satisfying~\eqref{beta0} and~\eqref{rho1}. For $\rho_1$ as in \eqref{rho1} and $R > 0$,  
 fix $\delta>0$ such that
 \begin{align}\label{delta:def}
\delta<\min\left\{ \frac{1}{\gamma + \beta_{B_{\rho_1}(0)}}, \frac{R}{8R + 2(N - 1)} \right\},
\end{align}
where $\beta_{B_{\rho_1}(0)}$ is given by~\eqref{beta0}. 
If $R = 0$, we take 
\begin{align}
\delta<\frac{1}{\gamma + \beta_{B_{\rho_1}(0)}} \,.
\end{align}

Then, there are measurable functions $\hat b_i$, $\hat c^{e}$, and a strong solution $\hat w_e$ of 
 \begin{equation}\label{weeq}
 \begin{aligned}
 (\hat w_e)_t - \Delta \hat w_e-\sum_{i=1}^N \hat b_i(x) (\hat w_e)_{x_i} - \hat c^e(x,t) \hat w_e &= 0\qquad \text { in }\Sigma_1\times (0,\infty),\\
 \hat w_e&= 0\qquad \text{ on }\partial \Sigma_1\times (0,\infty),\\
 \end{aligned}
\end{equation}
such that, for each $x\in\Sigma_1$ the sign of $w_e(x)$ is the same as the sign of $\hat w_e(x)$ and  for any subset $Q\subset \Sigma_1\times(0,\infty),$
\begin{align}\label{equivalence}
 \frac{1}{4}||\hat w_e||_{L^\infty(Q)}\leq ||w_e||_{L^\infty(Q)}\leq ||\hat w_e||_{L^\infty(Q)}.
\end{align}
Furthermore, for $x\in\Sigma_1$ and $R$ as in \eqref{R},
\begin{align}\label{bds}
\begin{aligned}
|\hat b_1(x_1)|&\leq 4\chi_{[0,\delta]}(x_1)+4\chi_{[R,R+\delta]}(|x|),\\
|\hat b_i(x_1)|&\leq 4\chi_{[R,R+\delta]}(|x|),\\
|\hat c^e(x)|&\leq |c^e|+\frac{2}{\delta}\chi_{[0,\delta]}(x_1)+8\chi_{[0,\delta]}(x_1)\chi_{[R,R+\delta]}(|x|) + \left(\frac{2}{\delta} + \frac{2(N - 1)}{R} \right)\chi_{[R,R+\delta]}(|x|) ,
\end{aligned}
\end{align}
where all terms containing $R$ are dropped if $R = 0$. 
In addition, 
\begin{align}\label{G}
 \hat c^e<-\gamma \quad \text{ in }\R^N\backslash G,\qquad 
G:=\left\{x\in\Sigma_1\::\:
\operatorname{dist}(x,\partial\Sigma_1)>\delta\text{ and }|x|<\rho_1\right\}.
\end{align}
\end{lemma}
\begin{proof}
Let $R\geq 0$ be as in \eqref{R}, $\delta>0$ be as in~\eqref{delta:def}, and recall that $I := [R, \infty)$. 

For $R > 0$ let $h : [0, \infty) \to\R$ and $g:[R,\infty)\to\R$ be given by
\begin{align*}
    h(t)&=\left(-\frac{1}{2\delta}\left(t-\delta\right)^2+\frac{\delta}{2}+\frac{1}{2}\right)\chi_{[0,\delta]}(t)+
    \left(\frac{\delta}{2}+\frac{1}{2}\right)\chi_{(\delta,\infty]}(t),\\
    g(r)&=
    \left(-\frac{1}{2\delta}\left(r-R-\delta\right)^2+\frac{\delta}{2}+\frac{1}{2}\right)\chi_{[R,R+\delta]}(r)
    + \left(\frac{\delta}{2}+\frac{1}{2}\right)\chi_{(R+\delta,\infty]}(r) \,.
\end{align*}
If $R = 0$, we keep the definition of $h$ and set $g = \frac{1}{2}$.  
Observe that $h$ and $g$ are differentiable, non-decreasing, concave,  piecewise $C^2$ functions with
\begin{align}\label{ulbgh}
 \frac{1}{2}\leq g<1\qquad \text{ and }\qquad \frac{1}{2} \leq h < 1\qquad \text{ in }\Sigma_1,
\end{align}
and
\begin{equation}\label{sdb}
\begin{gathered}
|h'(t)| \leq \chi_{[0,\delta]}(t)\,,  \qquad  |g'(r)| \leq  \chi_{[R,R+\delta]}(r)\,, \\ 
 |h''(t)| \leq \frac{1}{\delta} \chi_{[0,\delta]}(t)\,, \qquad 
 |g''(r)| \leq \frac{1}{\delta} \chi_{[R,R+\delta]}(r)\,.
\end{gathered}
\end{equation}

 Then, if $R > 0$  a direct calculation shows that
\begin{align*}
\hat w_e(x,t) := \frac{w_e(x,t)}{g(|x|)h(x_1)},\qquad (x,t)\in \Sigma_1\times(0,\infty),
\end{align*}
belongs to $C^2(\Sigma_1 \setminus S)$, where $S := \{x: x_1 = \delta \textrm{ or } |x| = R + \delta \}$ and $\hat w_e$
satisfies~\eqref{weeq} outside of $\Sigma_1 \setminus S$ with
\begin{align*}
    \hat b_i(x)&=
    \begin{cases}
    2\frac{h'(x_1)}{h(x_1)}+2\frac{g'(|x|)}{g(|x|)}\frac{x_1}{|x|}, &\text{ if }i=1,\vspace{.2cm}\\
    2\frac{g'(|x|)}{g(|x|)}\frac{x_i}{|x|}, &\text{ if }i\neq 1,
    \end{cases}\\
    \hat c^e(x)&=\frac{h''(x_1)}{h(x_1)}+2\frac{h'(x_1)}{h(x_1)}\frac{g'(|x|)}{g(|x|)}\frac{x_1}{|x|}+\frac{g''(|x|)+\frac{N-1}{|x|}g'(|x|)}{g(|x|)}+c^e(x) \,.
\end{align*}
Note that $|x| \geq R > 0$ and denominators do not vanish. If $R = 0$, $\hat w_e$ satisfies~\eqref{weeq} and $\hat b_i$, $\hat c^e$ as above with all terms containing $g'$ or $g''$ removed. In addition, $\hat w_e$ has bounded second derivatives on $\Sigma_1 \setminus S$ and, in particular, $\hat w_e$ belongs to $W^{1, \infty}(\Sigma_1)$. Thus, $\hat w_e$
is a strong solution of~\eqref{weeq}. 

By \eqref{ulbgh}, we have that \eqref{equivalence} holds and, by using \eqref{sdb} and $\frac{|x_i|}{|x|} \leq 1$, we obtain that \eqref{bds} holds. 
Furthermore, since $R + \delta < \rho_1$, $g' = g'' = 0$ on $[\rho_1, \infty)$ and $h$ is concave, it follows that
\begin{equation}
 \hat c^e < - \gamma \qquad \textrm{if } \quad |x| > \rho_1 \,.
\end{equation}
If $x_1 < \delta$ and $|x| < \rho_1$, then by \eqref{delta:def}
\begin{equation}
 \hat c^e \leq -\frac{1}{\delta}+ \left(8 + \frac{2(N - 1)}{R} - \frac{1}{\delta}  \right)\chi_{[R,R+\delta]}(|x|)+\sup\limits_{|x|<\rho_1}|c^e|  < - \gamma \,,
\end{equation}
where, as above, we drop the term with $\chi_{[R,R+\delta]}(|x|)$ if $R = 0$. Finally, in the case $R > 0$, if $x_1  > \delta$ and $|x| \in (R, R + \delta)$, we have, by \eqref{delta:def}, that
\begin{equation}
 \hat c^e \leq  8 + \frac{2(N - 1)}{R} - \frac{1}{\delta}  < - \gamma 
\end{equation}
and~\eqref{G} follows.
\end{proof}

Next, using $(f_2)'$, we construct a suitable subsolution inside $G$. Let $G$ be as in~\eqref{G} and $\gamma$ as in~\eqref{gammaf2}. Let $J = (a,b)$ be the interval given by $(f_2)'$ for $M=M_1$ (with $M_1$ as in $(U_2)$) and for some $R<a<b<\rho$.  Let $e_1=(1,0,\ldots,0)\in\R^N$.

 \begin{lemma}\label{subsolb} 
 Let $0\leq\tau<T\leq\infty$, $r=\frac{|I|}{2}$, $B=B_r(\frac{b+a}{2}e_1)\subset G$, and let $\eta$ denote the positive principal eigenfunction of the Laplacian with Dirichlet boundary conditions on the unit ball $B_1(0)$.  Then $\varphi(x,t):=e^{-\gamma t}\eta(\frac{x-x_0}{r})$ is a (strict) subsolution of~\eqref{weeq} in $B\times(\tau,T)$, namely,
 \begin{align}\label{phi:def}
  \varphi_t<\Delta \varphi+\sum_{i=1}^N \hat b_i\,\varphi_{x_i}+\hat c^e\,\varphi\  \text{ in }B\times(\tau,T),\quad \varphi=0\ \text{ on }\partial B\times(\tau,T) \,,
 \end{align}
 where $\hat b_i$ and $\hat c^e$ are defined in the proof of Lemma \ref{T:l}.
  \end{lemma}

 \begin{proof}
 Since $\hat c^e =  c^e$ and $\hat b_i^e = 0$ in $G$ it suffices to verify that 
 \begin{equation}
\varphi_t<\Delta \varphi + c^e\varphi \ \text{ in }B\times(\tau,T),\quad \varphi=0\ \text{ on }\partial B\times(\tau,T).
\end{equation} 
However, by the definition, $\varphi (x) = 0$ on $\partial B$ and, in addition, from~\eqref{gammaf2} (with $J=\{|x|\::\: x\in B\}$) we have that
\begin{equation}
\varphi_t - \Delta \varphi - c^e \varphi = \left(-\gamma + \frac{\lambda_1}{r^2} + c^e\right) \varphi<0 \, \quad \text{ in }B\times(\tau,T).
\end{equation}
\end{proof}

\subsection{Regularity of solutions}\label{regularity}

In this section we show that $u$ is locally equicontinuous.  We extend the proof of \cite[Lemma 3.1]{saldana:2016} to unbounded domains. Fix $\alpha\in(0,1]$ and a domain $\Omega\subset \R^N$. Set $Q:=\Omega\times(\tau,T)$ for $0\leq \tau<T.$ Following \cite[page 4]{QS07} we define $C^{1+\alpha,\frac{1+\alpha}{2}}(Q):=\{f : |f|_{1+\alpha;Q}<\infty \},$ where
\begin{equation}\label{holder:norms}
\begin{aligned}
&[f]_{\alpha;Q}:= \sup\bigg\{ \frac{|f(x,t)-f(y,s)|}{|x-y|^\alpha+|t-s|^\frac{\alpha}{2}}\::\: (x,t),(y,s)\in Q,\ (x,t)\neq (y,s)\bigg\},\\
&|f|_{1+\alpha;Q}:=\sum_{|\beta|\leq 1}\sup_{Q}|D_x^\beta f|+\sum_{|\beta|= 1}[D_x^\beta f]_{\alpha;Q},
 \end{aligned}
 \end{equation}
and $D^\beta_x$ denotes spatial derivatives of order $\beta\in \mathbb N_0^N$.

\begin{lemma}\label{regularity:lemma}
Let $u$ be a classical solution of~\eqref{model} and assume that $(U_1)$, $(U_2)$, $(f_0)$, and $(f_1)$ or $(f_1)'$ hold. If $\tilde{\gamma} =  \frac{1}{N + 3}$, then 
 for every $R_1>R$ there is $\tilde C>0$ satisfying
\begin{align*}
|u|_{1+\tilde\gamma;\overline{B_{R_1}(0)\cap\Sigma}\times[s,s+2]}\leq \tilde C \qquad \text{ for all }s>2.
\end{align*}
\end{lemma}

\begin{proof}
Let $R\geq 0$ be as in \eqref{R}, $M_1>0$ as in $(U_2)$, $K_{R_1}$ as in $(f_0)$, and $C=C(M_1, [R,R_1])$ as in $(f_1)'$ (or implicitly given by $(f_1)$). Then $|u|\leq M_1$ in $\Sigma\times(0,\infty)$ and 
\begin{align*}
 |f_u(t,r,v)|<C\qquad \text{ for }t>0,\ r\in [R,R_1],\ v\in[-M_1,M_1].
\end{align*}
Furthermore, $f(t,r,0)<K_{R_1}$ for $t>0$ and $r\in[R,R_1]$. Fix $s>2$. Then, $u$ satisfies that
\begin{align*}
 |u_t-\Delta u| &= |f(t,|x|,u)| = \left|\int_0^1 \partial_u f(t,|x|,su)ds\ u+f(t,|x|,0)\right|<M_1C+K_{R_1}=:C_1
\end{align*}
for $(x,t)\in Q:=(B_{R_1}(0)\cap\Sigma)\times(s,s+2)$.  Recall that, if $R>0$, then $u=0$ on $\partial B_R(0)$.  Then, by \cite[Theorem 7.22 or Theorem 7.30]{lieberman}, there is $C_2=C_2(R,C,M_1,K_{R_1},N,R_1)>0$ such that $\|D^2 u\|_{L^{N+3}(Q)}+\|u_t\|_{L^{N+3}(Q)}\leq C_2.$  By a standard interpolation argument (see \cite[Lemma 7.20]{lieberman}), $\|u\|_{W_{N+3}^{2,1}(Q)}\leq C_3$ for some constant $C_3=C_3(R,C,M_1,K_{R_1},N,R_1)>0.$  By Sobolev embeddings (see, for example, \cite[embedding $($1.2$)$]{QS07} and the references therein), we then have that $u\in C^{1+\tilde\gamma,(1+\tilde\gamma)/2}(Q)$ for 
$\tilde\gamma= \frac{1}{N+3}\ \in\ (0,1),$ and there is a constant $C_4=C_4(R,C,M_1,K_{R_1},N,R_1)>0$ such that $|u|_{1+\tilde{\gamma}; Q}\leq C_4 \|u\|_{W_{N+3}^{2,1}(Q)}\leq C_4 C_3=:\tilde C$. 
\end{proof}

From Lemma~\ref{regularity:lemma} and the uniform decay assumption $(U_1)$ it follows that $\omega(u)$ (as defined in~\eqref{omega}) is a nonempty and compact set in $C_0(\Sigma)$ and that~\eqref{omega:unif} holds.

\subsection{Estimates of solutions of linear equations}

In  this subsection we prove bounds for general linear parabolic equations needed for the symmetry results. Fix a domain $\Omega\subset \R^N$ (possibly unbounded), $0\leq \tau<T\leq \infty$, and denote $Q:=\Omega\times(\tau,T)$. 
For $i=1,\ldots,N,$ let $b_i,c:Q\to \R$ be measurable functions such that, for every bounded subdomain $\omega\subset \Omega$, there is $\beta_\omega>0$ such that 
\begin{align}\label{coeff}
|b_{i}|,|c|<\beta_\omega\qquad \text{ in }\omega\times(\tau,T).
\end{align}

We formulate the following Harnack inequality for sign-changing solutions proved in  \cite{polacik}. Let $v^-:=-\min\{0,v\}\geq 0$ and $v^+:=\max\{0,v\}\geq 0$ denote the negative and positive parts of $v$, respectively. We define $\partial_P Q$ to be the parabolic boundary of a cylindrical domain $Q=U\times (\tau,T),$ that is,  
$\partial_P Q = (\partial U \times (\tau,T))\bigcup (\bar{U}\times \{\tau\})$.

\begin{lemma}(\cite[Lemma 3.4]{polacik})  \label{harnack:inequality} 
Let $\Omega$ be a bounded domain. Given $d>0, \theta>0$ there is a positive constant $\kappa$ determined only by $N, \operatorname{diam}(\Omega), \beta_\Omega,d$, and $\theta$ with the following property.  If $D,U$ are domains in $\Omega$ with $D\subset \subset U,$ $\operatorname{dist}(\overline{D}, \partial U)\geq d,$ and $v$ is a 
 solution of
\begin{align*}
 v_t &=\Delta v+ \sum_{i=1}^Nb_i(x,t)v_{x_i} + c(x,t)v \qquad \text{ in }U\times (\tau,\tau+4\theta),
\end{align*}
then 
\begin{align*}
 \inf_{D\times (\tau+3\theta,\tau+4\theta)}v(x,t)\geq \kappa ||v^+||_{L^\infty(D\times(\tau+\theta,\tau+2\theta))}-e^{4m\theta}\sup_{\partial_P(U\times(\tau,\tau+4\theta))}v^-,
\end{align*}
where $m=\sup\limits_{U\times (\tau,\tau+4\theta)}c.$ 
\end{lemma}

The next lemma is the weak maximum principle in unbounded domains, see, for example, \cite[Lemma~2.1]{polacik:unbounded}, \cite{lieberman,protter}. 

\begin{lemma}\label{weakmp} Let $U\subset \R^N$ be a (possibly unbounded) domain and let $v$ be such that 
\begin{align*}
 v_t \geq \Delta v+\sum_{i=1}^N b_i(x,t)v_{x_i} + c(x,t)v\quad \text{ in }U\times(0,\infty)
\end{align*}
with
\begin{align}\label{m:def}
 m:=\sup\limits_{U\times(0,\infty)}c(x,t)<\infty
\end{align}
and
 \begin{align}\label{mpd}
  \lim_{|x|\to \infty, x \in U} v(x,t)=0\qquad \text{ for all }t>0.
 \end{align}
Then, for each $(x,t)\in U\times(0,\infty)$ one has
\begin{align*}
e^{-mt} v^{-}(x,t)\leq \sup_{(y,s)\in \partial_p (U\times(0,\infty))} e^{-ms} v^{-}(y,s).
\end{align*}
\end{lemma} 
Note that assumption \eqref{mpd} can be weakened, see for example \cite[Proposition~52.4]{QS07}.

\subsection{Geometric lemmas}\label{geometric:lemmas}

We use the following geometric characterization of foliated Schwarz symmetry proved in \cite{saldana-weth}. Recall that, for $e\in \S$, $\Sigma(e):=\{x\in \Sigma: x\cdot e>0\}$ and $\sigma_e(x):=x-2(x\cdot e)e$.

\begin{prop}\cite[Proposition 3.3]{saldana-weth} \label{sec:char-foli-schw-1}
Let $\cal U$ be a set of continuous functions defined on a radial domain $\Sigma\subset
 \mathbb R^N,$ $N\geq 2$ such that 
\begin{align*}
\cM:=\{e\in \S \mid z(x) \ge z(\sigma_e(x)) \text{ for all } x\in
 \Sigma(e) \text{ and } z \in {\cal U}\} \neq \emptyset \,.
\end{align*}
Fix $\tilde e\in \cM$ and assume that  for any two dimensional subspaces $P\subseteq
 \mathbb R^N $ containing $\tilde e$ there exist 
 $p_1 \neq \ p_2$ in a connected component of $\cM \cap P$ such that
 $z \equiv z \circ \sigma_{p_1}$ and $z \equiv z \circ
 \sigma_{p_2}$ for every $z \in \cal U$. Then, there exists $p
 \in \S$ such that every $z \in \cal U$ is foliated Schwarz symmetric with
 respect to $p$.
\end{prop}

Next, we show the following lemma, which is used to show Theorem \ref{theorem:unbounded}.

\begin{lemma}\label{lemma2.5}
 Let $U\subset \S$ be an open set, $\Sigma\subset \mathbb R^N$ a radial domain and $z\in C(\Sigma).$  If 
\begin{align}
 z(x)&\geq z(\sigma_e(x)) \hspace{1cm} \text{ for all } x\in \Sigma(e)\text{ and } e\in U,\label{h1}\\
 z&\equiv z\circ \sigma_{p} \hspace{1.3cm} \text{ for some } p\in U,\label{h2}
\end{align}
then $z$ is radially symmetric.
\end{lemma}

The proof of Lemma \ref{lemma2.5} is based on the following results. 

\begin{lemma}\cite[Lemma 3.1]{saldana-weth} \label{lemma1}
 Let $v\in C(\mathbb R)$ be an even and $2\pi$-periodic function, 
and denote $\cR := \{\eta \in \mathbb{R} : v(2\eta - \phi) = v (\phi) \textrm{ for each } \phi \in \mathbb{R}\}$  the set of  points of reflectional symmetry of
$v$. If for some $\eta \in \R$,  
 \begin{equation}
\label{geo4}     
   \begin{aligned}
  v(\eta + \varphi)&\geq v(\eta -\varphi)\ \  \ \text{ for all } \varphi\in
  [0,\pi] \text{ and }\\
  v(\eta + \varphi_0)&> v(\eta -\varphi_0)\ \  \text{ for some }\varphi_0\in(0,\pi),
   \end{aligned}
 \end{equation}
then  $\cR= \{n \pi \::\: n \in \mathbb{Z}\}$. 
\end{lemma}

\begin{lemma}\label{lemma2}
 Let $v\in C(\mathbb R)$ be a $2\pi$ periodic function. If 
 \begin{itemize}
  \item [$(i)$] (symmetry) 
  \begin{align}
   v(\theta)&= v(-\theta)\qquad \text{ for all }\theta\in\R, \label{g1}
  \end{align}
\item[$(ii)$] (local reflectional inequalities)  There is $\varepsilon>0$ such that
\begin{align}
   v(\eta+\theta)&\geq v(\eta-\theta)\qquad \text{ for all }\eta\in(0,\varepsilon),\ \theta\in(0,\pi) \,,\label{g2}
\end{align}
 \end{itemize}
then $v$ is nondecreasing in $(0,\pi)$ and also
\begin{align}
  v(\eta+\theta)&\geq v(\eta-\theta) \qquad \text{ for all }\eta,\theta\in(0,\pi).\label{g3}
 \end{align}
Moreover, if the inequality in~\eqref{g2} is strict, then~\eqref{g3} is also strict and $v$ is strictly increasing in $(0,\pi).$
\end{lemma}

\begin{proof}
If~\eqref{g2} holds with equality for all $\eta\in(0,\varepsilon)$ and $\theta\in(0,\pi)$, then $v$ is constant and the conclusion follows trivially. Now, assume that there is $\eta_0\in(0,\varepsilon)$ and $\theta_0\in(0,\pi)$ such that~\eqref{g2} is 
strict, then  Lemma~\ref{lemma1} implies that $0$ and $\pi$ are points of reflectional symmetry of $v$, that is, 
$v(\theta) = v(-\theta) = v(2\pi - \theta)$ for each $\theta$. 

Clearly~\eqref{g2} implies that $v$ is non-decreasing on $(0, \pi)$.  Indeed, fix any 
$\theta_1,\theta_2\in(0,\pi)$ such that $\alpha:=\theta_2-\theta_1\in(0,\varepsilon).$ Set $\eta:=\alpha/2\in(0,\varepsilon)$ and $\varphi_0:=\theta_1+\alpha/2\in(0,\pi).$ Then, 
$\theta_2=\eta+\varphi_0$ and $-\theta_1=\eta-\varphi_0$, and therefore by \eqref{g2}
\begin{align}\label{g5}
 v(\theta_2)=v(\eta+\varphi_0)\geq v(\eta-\varphi_0)=v(-\theta_1)=v(\theta_1).
\end{align}
Moreover, if the inequality in~\eqref{g2} is strict then $v$ is strictly increasing. 

To prove~\eqref{g3}, fix $\eta,\theta\in(0,\pi)$ and since $v$ is even, it suffices to prove $v(\eta+\theta) \geq v(|\eta-\theta|)$. If $\eta + \theta \leq \pi$, 
then $\eta+\theta \geq |\eta-\theta|$ and~\eqref{g3} follows from the monotonicity of $v$. If $\eta + \theta > \pi$, then $\pi \geq 2\pi - (\eta + \theta) \geq |\theta -\eta|$ and, by the symmetry and monotonicity of $v$, $v(\eta + \theta) = v(2\pi - (\eta + \theta))\geq v(|\theta-\eta|)$, as required. The case with strict inequality follows analogously. 
\end{proof}

\begin{figure}[h!]
\begin{center}
\begin{picture}(230,130)
    \put(0,5){\includegraphics[width=.50\textwidth]{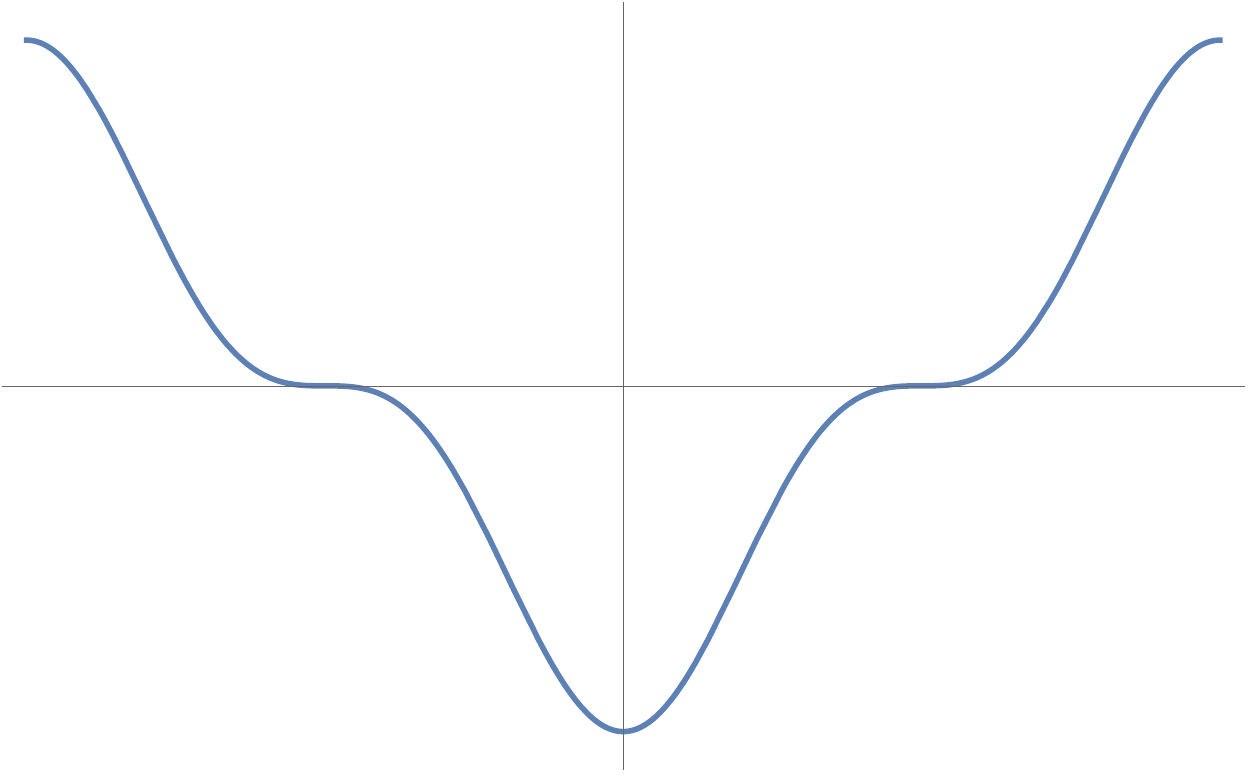}}
    \put(120,130){$v(\theta)$}
    \put(225,70){$\theta$}
    \put(212,60){$\pi$}
    \put(-10,60){$-\pi$}
  \end{picture}
  \end{center}
\caption{Example of a $2\pi$-periodic function $v$ with reflectional symmetry at 0 and satisfying the reflectional inequalities~\eqref{geo4} and~\eqref{g3}.}
\end{figure}
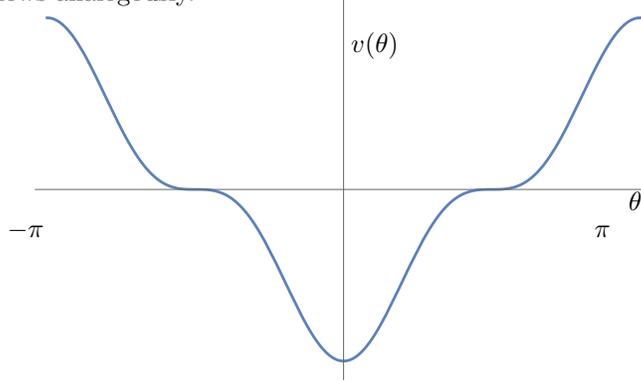
For $e \in \S$ let
\begin{align}\label{ze:def}
 z_e(x)&:= z(\Gamma_e(x))- z(\sigma_e(\Gamma_e(x))), \ \ \  x\in \Sigma_1, z\in \omega(u)\,,
\end{align}
where $\Gamma_e$ is a rotation that maps $e_1$ to $e$ (introduced in Section~\ref{linearization}).
Recall $\Sigma_1 := \{x\in \Sigma: x_1 > 0 \}.$

\begin{lemma}\label{lemma2:cor}
 Let $z\in C(\Sigma)$, $\eps\in(0,\pi)$, and let
 \begin{align*}
 e(\eta):=(\cos\eta,\sin\eta,0,\ldots,0)\in \R^N\qquad \text{for any $\eta\in(0,2\pi)$}.
 \end{align*}
If $z_{e(0)}\equiv 0$ in $\Sigma_1$ and $z_{e(\eta)}>0$ in $\Sigma_1$ for all $\eta\in(0,\eps)$ then
\begin{equation}\label{lem:c}
 z_{e(\eta)}>0\quad \text{ in }\Sigma_1\quad \textrm{ for all }\eta\in(0,\pi),\qquad 
 z_{e(\eta)}<0\quad \text{ in }\Sigma_1\quad \textrm{ for all }\eta\in(-\pi,0).
\end{equation}
\end{lemma}
\begin{proof}
 For every $x\in\R^N$ ($N\geq 2$) we write $x=(x_1,x_2,x')$ for some $x'\in \R^{N-2}$. Fix $x'\in \R^{N-2}$, $r>0$, and let $v:(0,2\pi)\to\R$ be given by
$v(\theta):=z(r\sin\theta,r\cos\theta,x').$  With this notation, \eqref{ze:def} can be rewritten for
 $x_1 = r \sin \theta$ and $x_2 = r \cos \theta$ as
\begin{equation}
z_{e(\eta)}(x) = v(\theta + \eta) - v(\eta - \theta).
\end{equation} 
Note that we are not using the usual polar coordinates. 
 Since $z_{e(0)}\equiv 0$ in $\Sigma_1$ and $z_{e(\eta)}>0$ in $\Sigma_1$ for all $\eta\in(0,\eps)$ we have that $v(\theta)=v(-\theta)$ for $\theta\in (0,\pi)$ and $v(\eta+\theta)>v(\eta-\theta)$ for all $\theta\in (0,\pi)$ and $\eta\in(0,\eps).$ By Lemma~\ref{lemma2} we have that $v(\eta+\theta)>v(\eta-\theta)$ for $\theta,\eta\in (0,\pi)$.  Using the symmetry with respect to 0 and the $(2\pi)$-periodicity we also have that $v(\eta-\theta)>v(\eta+\theta)$ for $\theta\in (0,\pi)$ and $\eta\in(-\pi,0)$. Since this holds for $x'\in \R^{N-2}$ and $r>0$ arbitrary we obtain that~\eqref{lem:c} holds.

\end{proof}

We are ready to show Lemma \ref{lemma2.5}.

\begin{proof}[Proof of Lemma \ref{lemma2.5}]
 Let $p\in U$ be as in~\eqref{h2}, $q\in \S\backslash\{p\}$ and $P:=\operatorname{span}\{q,p\}.$ Without loss of generality we assume that $p=(1,0,\ldots,0)$ and $P:=\{(x_1,x_2,0,\ldots,0)\mid x_1,x_2\in\mathbb R\}.$ Using polar coordinates, let $\tilde v(r,\varphi,x') := z(r\cos\varphi,r \sin \varphi, x')=z(x) $ with $x'=(x_3,\ldots,x_{N}) \in \R^{N-2},$ $(x_1,x_2,x')=x\in \Sigma,$ $\varphi\in\mathbb R,$ and $r=|x|$.  Fix $r>0,$ $x' \in \R^{N-2}$ and define $v: \R \to \R,$ $v(\varphi):=\tilde v (r_,\varphi,x').$  Since $U$ is open, there exists $\varepsilon>0$ such that
 $v(\varphi)=v(-\varphi)$ for $\varphi\in \mathbb R$ and $v(\eta+\varphi)\geq v(\eta-\varphi)$ for $\eta \in (-\varepsilon,\varepsilon)$ and $\varphi\in (0, \pi)$.  Then, by Lemma~\ref{lemma2} we have that $v(\eta+\varphi)\geq v(\eta-\varphi)$ for $\eta \in (-\varepsilon,\pi),\varphi\in (0, \pi),$ which yields that $v(\eta+\varphi)= v(\eta-\varphi)$ for $\eta \in (-\varepsilon,0),\varphi\in (0, \pi),$  and consequently $v$ is constant, that is, $\tilde v$ does not depend in $\varphi$.  Since  $r>0$ and $x' \in \R^{N-2}$ 
were arbitrary, we have $z\equiv z\circ \sigma_{e}$ for all $e\in P$, and in particular $z\equiv z\circ \sigma_{q}.$  Since $q$ was chosen arbitrarily in $\S$  the radial symmetry follows.
\end{proof}

\section{Symmetry results}\label{symmetry:results}

\subsection{Strong stability outside compact sets}

In this section we assume that $(f_0)$, $(f_1)'$, $(f_2)'$ hold and  $u$ is a classical solution of~\eqref{model} satisfying $(U_0)$--$(U_2)$. We recall some previously used notation.  Let $\Sigma = \R^N\setminus B_R$, where  $B_R = \{x \in \R^N : |x| <  R\}$ if $R>0$ and $B_0 := \emptyset$; furthermore, as defined in the introduction, $\Sigma_1=\Sigma(e_1)=\{x\in \Sigma\::\: x\cdot e_1>0\}$, with $e_1=(1,0,\ldots,0)\in\R^N$.  Let $G$ be given by~\eqref{G} and $U$ be a bounded domain such that 
 \begin{align}\label{U}
 G\subset\subset U \subset \subset\Sigma_1. 
 \end{align}
  If $J=(a,b)$ is the interval given in~\eqref{Idef}, by $(f_2)'$, then 
  \begin{align}\label{B}
  B=B_{\frac{b-a}{2}}\Big(\frac{a+b}{2}e_1\Big)\subset\subset G
  \end{align}
 as in Lemma~\ref{subsolb}, see Figure~\ref{fig1}.  For $e\in\S$, let $\hat w_e$ be as in Lemma~\ref{T:l} and note that, by~\eqref{bds}, there is $\hat\beta_U>0$ (independent of $e$) such that
  \begin{align}\label{hbeta}
   |\hat c^e|+|\hat b_i|<\hat\beta_U\qquad \text{ in }U\times(0,\infty).
  \end{align}

\begin{figure}[h!]
\begin{center}
\includegraphics[height=4cm]{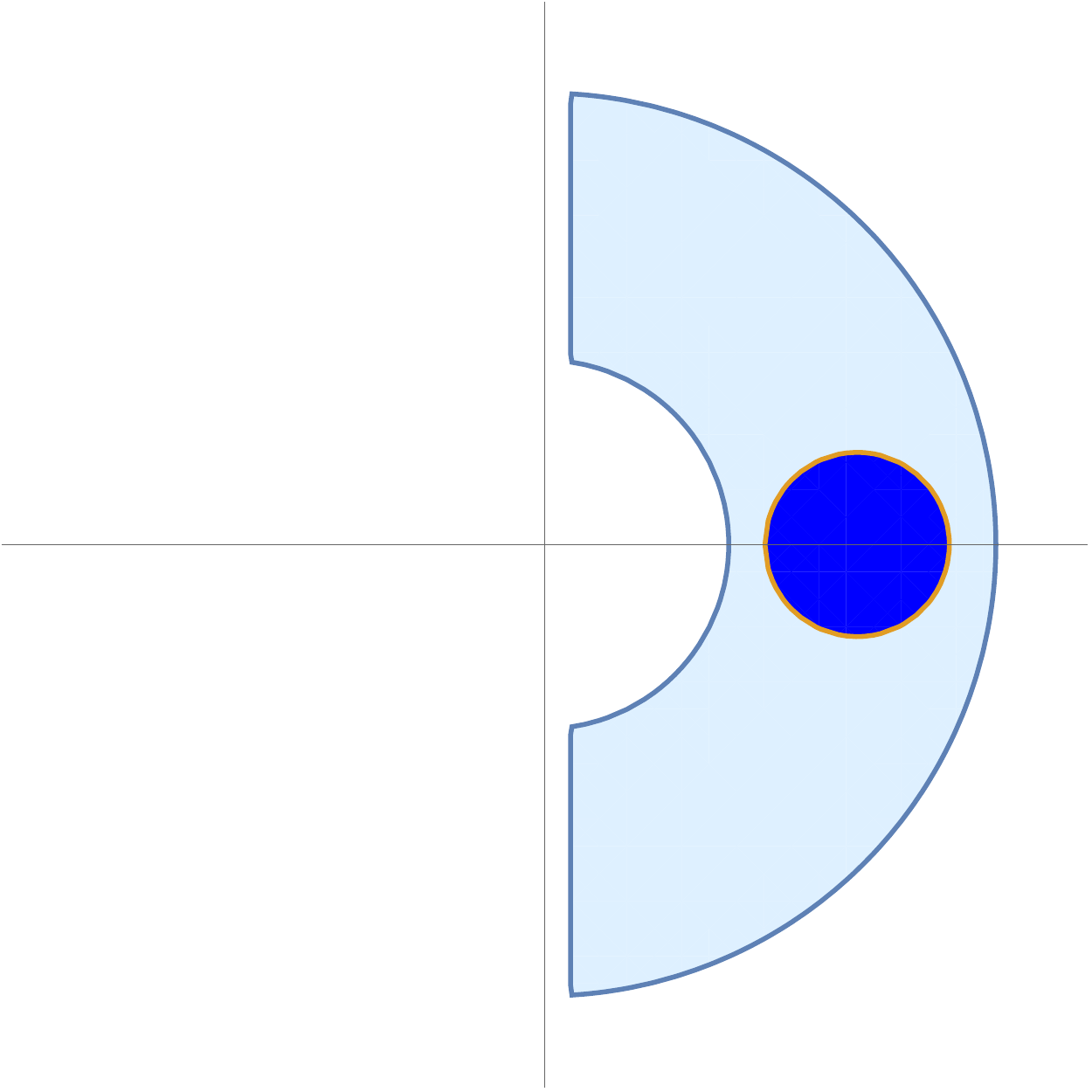}
\caption{The set $G$ with $\Sigma=\R^2\backslash B_R$ and the ball $B=B_\frac{b-a}{2}(\frac{a+b}{2}e_1)$.}\label{fig1}
\end{center}
\end{figure}
  
The next proposition extends the strategy in \cite[Theorem 3.7]{polacik} to unbounded domains.

\begin{prop}\label{polaciku}
 Let $\gamma$ be as in~\eqref{gammaf2}. There is $\mu>0$ such that, if there is $\tau>0$ and $e\in\S$ with
\begin{align}
\hat w_e& >0,\qquad \text{ in }\overline{G} \times [\tau,\tau + 1),\nonumber\\
||\hat w_e^-(\cdot,\tau)||_{L^\infty(\Sigma_1\backslash G)}& \leq \mu ||\hat w_e||_{L^\infty(G\times (\tau+\frac{1}{8},\tau+\frac{1}{4}))},\label{obs0}
\end{align}
then $w_e>0$ in $\overline{G}\times [\tau,\infty)$ and $\lim_{t\to \infty}|w_e^-(\cdot,t)||_{L^\infty(\Sigma_1)}=0$.
\end{prop}

\begin{proof}
For $\theta=\frac{1}{8}$, $D=B$ as in \eqref{B}, $\Omega=G$, and $U$ as in~\eqref{U}, $\hat\beta_U$ as in~\eqref{hbeta}, and $\gamma$ as in~\eqref{gammaf2}, let $\kappa>0$ be given by Lemma~\ref{harnack:inequality}, and let $\mu>0$ be such that
\begin{align}\label{mu}
 \mu<\frac{\kappa^2}{\kappa+1}<\kappa.
\end{align}
We argue by contradiction. Assume (by the maximum principle) that there is $T\geq \tau+1$ such that
\begin{align}
 \hat w_e&>0\qquad \text{ in }G\times [\tau,T),\nonumber\\
 \hat w_e(x^*,T)&=0\qquad \text{ for some } x^*\in \partial G. \label{cont}
\end{align}
Define $r_0:=||(\hat w_e)^{-}(\cdot,\tau)||_{L^\infty(\Sigma_1\backslash G)}.$  By Lemma~\ref{T:l}, $m:=\sup\limits_{\Sigma_1\backslash G \times (\tau,\infty)}c^e\leq -\gamma$ and $\hat w_e\geq 0$ in $\partial(\Sigma_1\backslash G)\times [\tau,T].$  Then, by Lemma~\ref{weakmp},
\begin{align}
||(\hat w_e)^{-}(\cdot,t)||_{L^\infty(\Sigma_1)}&=||(\hat w_e)^{-}(\cdot,t)||_{L^\infty(\Sigma_1\backslash G)}\\
&\leq e^{mt}\sup_{(s,y)\in\partial_p (\Sigma_1\backslash G \times [\tau,T])} e^{-m s} (\hat w_e)^{-}(y)\nonumber\\
&=e^{m(t-\tau)} ||(\hat w_e)^{-}(\cdot,\tau)||_{L^\infty(\Sigma_1\backslash G)}\label{obs1} \\
&\leq r_0,\label{obs}
\end{align}
for all $t\in[\tau,T).$  Hence,~\eqref{obs0},~\eqref{mu},~\eqref{obs}, and Lemma~\ref{harnack:inequality} imply, for $(x,t)\in \overline{G}\times [\tau+3\theta,\tau+4\theta]$, that 
\begin{align}\label{obs3}
 \hat w_e(x,t)&\geq \kappa ||(\hat w_e)^+||_{L^\infty(G\times(\tau+\theta,\tau+2\theta))}-e^{4m\theta}\sup_{\partial_P(U\times(\tau,\tau+4\theta))}\hat w_e^-\nonumber\\
&\geq \kappa ||(\hat w_e)^+||_{L^\infty(G\times(\tau+\theta,\tau+2\theta))}-e^{4m\theta}r_0
\geq r_0 \bigg(\frac{\kappa}{\mu}-1\bigg)=:r_1>0 \,. 
\end{align}
   Furthermore, by Lemma~\ref{subsolb}, there is $\varphi(x,t)=e^{-\gamma t}\eta(\frac{x-x_0}{r})$ satisfying~\eqref{phi:def} (with $x_0=\frac{a+b}{2}e_1$). By~\eqref{obs3},
\begin{align*}
\hat w_e&>0\qquad \text{ in } G\times [\tau+4\theta,T),\\
\varphi &= 0\qquad \text{ on } \partial B\times [\tau+4\theta,T),\\
\hat w_e(x,\tau+4\theta)&\geq r_1\frac{\varphi(x,\tau+4\theta)}{||\varphi(\cdot,\tau+4\theta)||_{L^\infty(B)}}\qquad \text{ for }x\in B.
\end{align*}
Then, by comparison,
\begin{align*}
 \hat w_e(x,t)\geq r_1\frac{\varphi(x,t)}{||\varphi(\cdot,\tau+4\theta)||_{L^\infty(B)}}\geq r_1 e^{-\gamma(t-\tau-4\theta)}\frac{\eta\left(\frac{x - x_0}{r}\right)}{||\eta||_{L^\infty(B)}}
\end{align*}
for $x\in B$ and $t\in[\tau+4\theta,T].$ Using Lemma~\ref{harnack:inequality} in $[\tau+4\theta,T]$ and~\eqref{obs1},
\begin{align*}
 \hat w_e(x,T)&\geq \kappa||\hat w_e||_{L^\infty(G\times(T-3\theta,T-2\theta))}-e^{4m\theta}\sup_{\partial_P(U \times(T-4\theta,T))}\hat w_e^-\\
&\geq \kappa r_1 e^{-\gamma(T-\tau-4\theta)}-e^{4m\theta}e^{m(T-\tau-4\theta)}r_0\\
&= r_0e^{-\gamma(T-\tau-4\theta)}\bigg(\kappa \bigg(\frac{\kappa}{\mu}-1\bigg) -e^{4m\theta}e^{(m+\gamma)(T-\tau-4\theta)}  \bigg)\\
&\geq r_0e^{-\gamma(T-\tau-4\theta)}\bigg(\frac{\kappa^2}{\mu}-\kappa  -1\bigg)>0,
\end{align*}
for all $x\in \overline{G}$, which contradicts~\eqref{cont}. Therefore, $\hat w_e>0$ in $\overline{G}\times [\tau,\infty)$, and consequently
by letting $t\to\infty$ in~\eqref{obs1}, we have  that 
 $\lim_{t\to \infty}||\hat w_e^-(\cdot,t)||_{L^\infty(\Sigma_1)}=0$.
\end{proof}

 For $e \in \S$, let $z_e$ be as in~\eqref{ze:def} and define
\begin{align}\label{cM:def}
 {\cal M}:=\{e\in \S\::\: z_e(x)\geq 0\text{ in }\Sigma_1 \text{ for all }z\in\omega(u)\}.
\end{align}

We are ready to show Theorem~\ref{strong:thm} via a rotating-plane method.  We split the proof in several lemmas.

\begin{lemma}\label{teopoho}
 Let $e\in \S$ be as in $(U_0)$ and $P\subset \mathbb R^N$ be a two dimensional subspace such that $e\in P$. Then, there exists 
 $\varepsilon>0$ such that $e'\in{\cal M}$ for all $e'\in \S\bigcap P$ with $|e'-e|<\varepsilon.$
\end{lemma}
\begin{proof} 
Due to rotational invariance, we can assume without loss of generality that $e=e_1=(1,0,\ldots,0)$ and $P=\{(x_1,0,\ldots,0,x_N)\mid x_1,x_N\in \mathbb R\}.$  By Lemma~\ref{weakmp},
\begin{align*}
e^{-mt} w_e^{-}(x,t)\leq \sup_{(y,s)\in \partial_P (\Sigma_1\times(0,\infty))} e^{-ms} w_e^{-}(y,s)=\sup_{\Sigma_1} w_e^{-}(\cdot,0)=0,
\end{align*}
for all $(x,t)$ in $\Sigma_1\times(0,\infty)$ and with $m:=\sup_{\Sigma_1\times(0,\infty)}c^e$.  Then $w_e\geq 0$ in $\Sigma_1\times(0,\infty)$ and, since $w_e(x,0)\not\equiv 0,$  the strong maximum principle (see, for example, \cite[Theorem 5 on page 173 and Remark 2 on page 175]{protter}) implies  that
\begin{align}\label{obs4}
w_e(x,t)>0\quad  \text{ in }\Sigma_1\times (0,\infty).
\end{align}

By~\eqref{obs4}, there is $\eta>0$ such that $\hat w_e>2\eta>0$ in $\overline{G}\times [1, 2],$ where $\hat w_e$ is as in Lemma~\ref{T:l}.  Then, by continuity and $(U_1)$, 
there exists $\varepsilon>0$ such that any $e'\in \S\bigcap P$ with $|e-e'|<\varepsilon$ satisfies that
\begin{align*}
 \hat w_{e'}> \eta>0\quad \text{ in }\overline{G}\times [1, 2],\qquad 
||\hat w_{e'}^-||_{L^\infty(\Sigma(e_1)\times (1,2))}\leq \mu\eta\leq \mu ||\hat w_{e'}||_{L^\infty(G\times (1, 2))},
\end{align*}
where $\mu$ is given by Proposition~\ref{polaciku}. Then, by Proposition~\ref{polaciku}, $\lim_{t\to \infty}||\hat w_{e'}^-(\cdot,t)||_{L^\infty(\Sigma(e'))}= 0,$ that is, $e'\in \cM$.
\end{proof}

\begin{lemma}\label{lobolema}
Let $e\in {\cal M}$ and let $P\subset \mathbb R^N$ be a two dimensional subspace with $e\in P.$ If there is $\tilde z\in\omega(u)$ such that $\tilde z_e\not\equiv 0,$ then there exists 
$\varepsilon>0$ such that $e'\in{\cal M}$ for all $e'\in \S\bigcap P$ with $|e-e'|<\varepsilon.$
\end{lemma}

\begin{proof}  
Recall that $U$ and $G$ were already defined at the beginning of the section. 
As above, we can assume without loss of generality that $e=e_1=(1,0,\ldots,0)$ and $P=\{(x_1,0,\ldots,0,x_N)\mid x_1,x_N\in \mathbb R\}$ 
and we use the same notation as in  Lemma~\ref{teopoho}.

 Since, by assumption, $\tilde z_e\not\equiv 0$, there is $\alpha>0$ and $x_0\in\Sigma_1$ such that $\tilde z_e(x_0)\geq 4\alpha>0.$ Without loss of generality we can assume that $x_0\in G=G(\rho_1,\delta)$ (making $\rho_1$ bigger or $\delta$ smaller if necessary). Fix $\{t_n\}_{n=0}^\infty\subset \mathbb R$ so that $\lim\limits_{n\to\infty}u(x,t_n)=\tilde z(x)$ for all $x\in\mathbb R^N.$  

There exists  ${\cal N}\in\mathbb N$ such that
\begin{align*}
 \hat w_e(x_0,t_n)&>\alpha\qquad \text{for all $n\geq {\cal N}$}.
\end{align*}

Since $z_e\geq 0$ in $\Sigma_1$ for all $z\in\omega(u)$ and $\lim\limits_{t\to\infty}\operatorname{dist}_{C_0(\mathbb R^N)}(\omega(u), u(\cdot,t))=0$, there is $T>0$ such that
\begin{align*}
 ||\hat w^-_{e}(x,t)||_{L^\infty(\Sigma_1)}<\frac{\mu\kappa \alpha}{8}e^{-4 \hat\beta_U}\qquad \text{ for all $t>T$},
\end{align*}
where $\kappa=\kappa(N,\operatorname{diam}(U),\hat\beta_U,\operatorname{dist}(U,G),1)>0,$ is given by Lemma~\ref{harnack:inequality}, $\mu$ as in Proposition~\ref{polaciku}, and $\hat\beta_U$ as in \eqref{hbeta}. 

Set $\tau:=t_{n_0}-1,$ with $n_0\geq {\cal N}$ and $t_{n_0}\geq T.$  Then, by Lemma~\ref{harnack:inequality},
\begin{align*}
 \inf_{G\times (\tau+3,\tau+4)}\hat w_e(x,t)&\geq \kappa ||\hat w_e^+||_{L^\infty(D\times(\tau+1,\tau+2))}-e^{4\hat\beta_U}\sup_{\partial_P(U\times(\tau,\tau+4))}\hat w_e^-,\\
&\geq \kappa\alpha-\frac{\mu\kappa \alpha}{8}\geq\frac{\kappa \alpha}{2}=:\eta>0.
\end{align*}
Then, by continuity and uniform decay (see $(U_1)$), there is $\varepsilon>0$ such that, for all $e'\in \S\bigcap P$ with $|e-e'|<\varepsilon,$
\begin{align*}
 ||\hat w_e(x,t)-\hat w_{e'}(x,t)||_{L^\infty(\Sigma_1)}<\frac{\eta\mu}{4}\quad  \text{ for all }t\in[\tau,\tau+4].
\end{align*}
Hence, if $\tilde \tau:=\tau+3,$
\begin{align*}
\inf_{D\times (\tilde \tau,\tilde \tau+1)}\hat w_{e'}(x,t)&\geq\frac{\eta}{2}>0,\\
||\hat w^-_{e'}(x,\tilde \tau)||_{L^\infty(\Sigma_1)}
&\leq ||\hat w^-_e(x,\tilde \tau)-\hat w^-_{e'}(x,\tilde \tau)||_{L^\infty(\Sigma_1)}+||\hat w^-_{e}(x,\tilde \tau)||_{L^\infty(\Sigma_1)}\\
&< \frac{\eta\mu}{4}+\frac{\mu\kappa \alpha}{8}e^{-4\hat\beta_U } < \frac{\mu\eta}{4}+\frac{\mu\eta}{4}=\mu\frac{\eta}{2}\leq \mu ||\hat w^+_{e'}(x,t)||_{L^\infty({D\times (\tilde \tau,\tilde \tau+1)})}.
\end{align*}
In particular, the hypotheses of Proposition~\ref{polaciku} are satisfied, and we have that $\lim_{t\to \infty}||\hat w^-_{e'}(x,t)||_{L^\infty(\Sigma_1)}=0,$ which yields the result.
\end{proof}

\begin{proof}[Proof of Theorem~\ref{strong:thm}]
  Let $e\in \S$ be as in $(U_0)$. Then, by Lemma~\ref{teopoho}, there is some $\varepsilon>0$ such that 
\begin{align}
e'\in {\cal M}\qquad \text{ for all }|e'-e|<\varepsilon.\label{Mopen}
\end{align}

 Let $P$ be any $2D$-plane that contains $e$ and the origin. Without loss of generality, we may assume that $e=(1,0,\ldots,0)$ and $ P = \{ x=(x_1,x_2,0,\ldots,0) \mid x_1 , x_2 \in \mathbb R \}.$  Define
  \begin{eqnarray*}
  e_\theta&:=&(\cos(\theta),\sin(\theta),0,\ldots,0),\\
  w_{\theta}(x,t)&:=&w_{e_\theta}(x,t),\ \  (x,t)\in \Sigma_1\times(0,\infty),\\
 z_{\theta}(x)&:=&z_{e_\theta}(x),\ \  x\in \Sigma_1,
   \end{eqnarray*}
  for $\theta \in [-\pi,\pi],$ and
  \begin{eqnarray*}
 \Theta_1&:=&\sup\{\theta\in[0,\pi) \mid e_\phi\in {\cal M} \text{ for all } 0\leq \phi\leq \theta \},\\
 \Theta_2&:=&\inf\{\theta\in(-\pi,0] \mid e_\phi\in {\cal M} \text{ for all } \theta \leq \phi\leq 0 \}.
   \end{eqnarray*}
  
 We show that the assumptions of  Proposition~\ref{sec:char-foli-schw-1} are satisfied.  Clearly $e_{\Theta_1},e_{\Theta_2}\in {\cal M}\cap P$ and, by Lemma~\ref{lobolema}, necessarily $z_{\Theta_1}\equiv z_{\Theta_2}\equiv 0$ for all $z\in\omega(u).$  Then, $H(e_{\Theta_1})$ and $H(e_{\Theta_2})$ are symmetry hyperplanes for all  elements in $\omega(u)$ and, by~\eqref{Mopen}, $\Theta_2<\Theta_1$.
 
 If $\Theta_2 \neq -\pi $ or $\Theta_1 \neq \pi$, then $e_{\Theta_1}\neq e_{\Theta_2}$ and $e_\varphi\in\{e\in \S\mid z_e \geq 0 \text{ in } \Sigma_1\}$ for all $\varphi\in(\Theta_2,\Theta_1)$ by the definition of ${\cal M}.$   In this case, let $p_1=e_{\Theta_1}$ and $p_2=e_{\Theta_2}$.
 
 If $\Theta_2 = -\pi $ and $\Theta_1 = \pi$, then 
(since $z_e=-z_{-e}$) we have that $z_{e_\varphi} \equiv 0$ for all $\varphi\in(-\pi,\pi)$ and for all $z\in\omega(u)$. In this case, let $p_1=e_{\pi/2}$ and $p_2=e_{-\pi/2}$.

Since $P$ is arbitrarily chosen, the claim now follows from Proposition~\ref{sec:char-foli-schw-1} using these choices of $p_1$ and $p_2$.
\end{proof}

\begin{proof}[Proof of Theorem \ref{thm1}]
Let $\Sigma,$ $\eta,$ $p$, $a,$ $b$, $\alpha$, and $\beta$ as in the statement.  We begin by showing that $u$ satisfies the uniform decay condition $(U_1)$ in $\Sigma\times(1,\infty)$.  Let $0<\gamma <\min\{\beta,1\},$ 
\begin{align}\label{r1}
 r_1>\max\left\{
 \left(\frac{2}{\eta}(\gamma^2+1)\right)^\frac{1}{\beta-\gamma}
 ,\left(\frac{2}{\eta}
\|a\|_\infty\|u\|_\infty^{p-1}
\right)^\frac{1}{\beta-\alpha}
 ,1\right\},
\end{align}
 and $v(x,t):=\|u\|_\infty e^{r_1^\gamma-|x|^\gamma t}$.  Then, 
$v_t-\Delta v + (b(t)|x|^\beta-a(t)|x|^\alpha|u|^{p-1}) v=c(x,t)v$,  where, for $(x,t)\in \Sigma\backslash B_{r_1}(0)\times (0,1)$,
\begin{align*}
 c(t,x)&:=-|x|^{\gamma }
-\gamma^2 t^2 |x|^{2\gamma-2}
+\gamma  t (\gamma +N-2)|x|^{\gamma-2}
+b(t)|x|^{\beta}-a(t)|x|^\alpha|u|^{p-1}
\\
&\geq 
|x|^{\beta}
\left(
\eta
-\frac{\gamma^2}{r_1^{\beta+2(1-\gamma)}}
-\frac{1}{r_1^{\beta-\gamma}}
-\frac{\|a\|_\infty\|u\|_{\infty}^{p-1}}{r_1^{\beta-\alpha}}
\right)\geq  
|x|^{\beta}
\left(
\eta
-\frac{\gamma^2+1}{r_1^{\beta-\gamma}}
-\frac{\|a\|_\infty\|u\|_{\infty}^{p-1}}{r_1^{\beta-\alpha}}
\right)>0,
\end{align*}
by \eqref{r1}.  Note that 
\begin{align}\label{ub}
|u(t,x)|\leq \|u\|_\infty
\leq \|u\|_\infty e^{r_1^\gamma -|x|^\gamma t}=v(x,t)
\quad \text{ in $(\overline{B_{r_1}(0)\cap \Sigma})\times[0,1].$}
\end{align}
As a consequence, $\widetilde v:=v-u$ satisfies that
$\widetilde v_t-\Delta \widetilde v +(b(t)|x|^\beta -a(t)|x|^\alpha|u|^{p-1})\widetilde v
=c\, v >0$ in $(\Sigma\backslash B_{r_1}(0))\times(0,1]$ and $\widetilde v\geq 0$ on the parabolic boundary $\partial_P(\Sigma\backslash B_{r_1}(0)\times (0,1))$.  Then, by the maximum principle \cite[Proposition~52.4]{QS07}, $\widetilde v>0$ in $\Sigma\backslash B_{r_1}(0)\times (0,1)$, and by \eqref{ub}, $\widetilde v \geq 0$ in $\Sigma\times[0,1]$.  Arguing similarly with $\widetilde v=v+u$ yields that
\begin{align}\label{s1}
 |u(x,t)|\leq  \|u\|_\infty e^{r_1^\gamma}e^{-|x|^\gamma t}\quad \text{ in }\Sigma\times[0,1].
\end{align}
For $(x,t)\in (\Sigma\backslash B_{r_1})\times[1,\infty)$, let $\widehat v(x,t) := v(x,1)-u(x,t)$. Then, for $(x,t)\in (\Sigma\backslash B_{r_1})\times[1,\infty)$,
\begin{align*}
\widehat v_t-\Delta \widehat v + (b(t)|x|^\beta -a(t)|x|^\alpha|u|^{p-1})\widehat v
&=|x|^{\beta}
\left(
-\frac{\gamma^2}{|x|^{\beta+2(1-\gamma)}}
+\frac{\gamma(\gamma+N-2)}{|x|^{\beta-\gamma+2}}
+b(t)
-\frac{a(t)|u|^{p-1}}{|x|^{\beta-\alpha}}
\right)v(x,1)\\
&\geq |x|^{\beta}\left(
\eta-\frac{\gamma^2}{r_1^{\beta-\gamma}}
-\frac{\|a\|_\infty\|u\|_\infty^{p-1}}{r_1^{\beta-\alpha}}
\right)v(x,1)>0,
\end{align*}
by \eqref{r1}.  On the other hand, 
\begin{align}\label{b1}
|u(t,x)|\leq \|u\|_\infty\leq  \|u\|_\infty e^{r_1^\gamma-|x|^\gamma}=v(x,1)\quad \text{ in }(\overline{B_{r_1}(0)\cap\Sigma})\times[1,\infty).   
\end{align}
By \eqref{s1} and \eqref{b1}, $\widehat v>0$ on $\partial_P((\Sigma\backslash B_{r_1}(0))\times (1,\infty))$; and then, by the maximum principle \cite[Proposition~52.4]{QS07} and \eqref{b1}, $\widehat v>0$ in $\Sigma\times(1,\infty)$. Arguing similarly with $\widehat v=v+u$ yields that
\begin{align*}
 |u(x,t)|\leq  Me^{-|x|^\gamma}\quad \text{ in }\Sigma\times[1,\infty).
\end{align*}
As a consequence, $u$ satisfies the uniform decay assumption $(U_1)$ in $\Sigma\times(1,\infty)$. Note also that, by the assumption \eqref{U0intro}, the linearization procedure \eqref{Pie}, and the maximum principle \cite[Proposition 52.4]{QS07}, one has that $(U_0)$ is satisfied by $u(\cdot,t)$ for all $t>0$; in particular, $u(\cdot,1)$ satisfies $(U_0)$ and $(U_1)$. 

Then, by Theorem \ref{strong:thm}, it suffices to show that $(f_0)$, $(f_1)'$, and $(f_2)'$ are satisfied. Let 
 \begin{align*}
f(t,r,u):=a(t)r^\alpha|u|^{p-1}u-b(t)r^\beta u,\qquad 
f_u(t,r,u):=p\, a(t)r^\alpha|u|^{p-1}-b(t)r^\beta.
 \end{align*}
Then $f(t,s,0)=0$ and $(f_0)$ is satisfied. Assumption $(f_1)'$ also holds, because for $K>0$ and $k>1$,
 \begin{align*}
\sup_{r\in [0,k], t>0, u\in [-K,K]}|f_u(t,r,u)|
\leq pK^{p-1}k^\alpha\|a\|_\infty+k^\beta\|b\|_\infty =: C,
 \end{align*}
 where $C>0$ depends only on $K$, $k$, and the fixed data of the problem $p$, $\alpha$, $\|a\|_\infty$, $\|b\|_\infty$, and $\beta$. 

 To prove $(f_2)'$, note that for any $M > 0$
\begin{align*}
\max_{r\in [1,2],u\in[-M,M]}
 |f_u(t,r,u)|
 \leq 2^\beta \max\{\|a\|_\infty,\|b\|_\infty\} M^{p-1}=:C_M.
\end{align*}
Define $\eps:=\left(\frac{\eta}{p\max\{\|a\|_\infty,1\}}\right)^\frac{1}{p-1}$, and let $\rho_M>1$ be such that $\eta(\rho_M^\alpha-\rho_M^\beta)\leq -C_M-4\lambda_1$. Then, since $\alpha < \beta$
\begin{align*}
\max_{r>\rho_M,u\in[-\eps,\eps]}f_u(t,r,u)
&=\max_{r>\rho_M,u\in[-\eps,\eps]}p|u|^{p-1}a(t)r^\alpha-b(t)r^\beta\\
&\leq \max_{r>\rho_M,u\in[-\eps,\eps]}p\eps^{p-1}\|a\|_\infty r^\alpha-\eta r^\beta \leq \eta\max_{r>\rho_M}(r^\alpha- r^\beta)=\eta(\rho_M^\alpha-\rho_M^\beta)\\
&\leq-C_M-4\lambda_1
<-\max_{r\in [1,2], u\in[-M,M]}|f_u(t,r,u)|-4\lambda_1,
\end{align*}
and $(f_2)'$ holds. The result now follows from Theorem \ref{strong:thm}.
\end{proof}

\subsection{Weak stability outside compact sets}

Let $u$ be a classical solution of~\eqref{model} and assume that $(U_0)$--$(U_2)$, $(f_0)$, and $(f_1)$ from Theorem~\ref{unbounded} hold.  For any $e \in \S$, define $w_e$ as in Section~\ref{linearization}, let $z_e$ be as in~\eqref{ze:def}, and $\cM$ as in~\eqref{cM:def}.  

\begin{lemma}\label{lemma3}
If $e \in \cM$ and $\bar z\not\equiv 0$ for some $\bar z\in \omega(u),$ then $\bar z_e>0$ in $\Sigma.$  
\end{lemma}
\begin{proof}
Let $\{t_n\}\subset [0,\infty)$  such that $t_n\to\infty$ as $n\to \infty$ and $u(x,t_n)\to \bar z(x)$ for all $x\in\mathbb R^N.$ By hypothesis there exist $x_0\in\Sigma_1$ and $\alpha>0$ such that $\bar z_e \geq 4\alpha$. Then  there is ${\cal N}\in\mathbb N$ such that
\begin{equation*}
 w_e(x_0,t_n)>2\alpha>0,\hspace{1cm} \text{ for all } n\geq \cal N.
\end{equation*}
By Lemma~\ref{regularity:lemma}, there exists $\theta>0$ such that
\begin{equation*}
 w_e(x_0,s)>\alpha,\hspace{1cm} \text{ for all } s\in[t_n-4\theta,t_n]  \text{ and } n\geq {\cal N}.
\end{equation*}
Fix $D$ and $U$ such that $x_0\in D\subset\subset U\subset\subset \Sigma.$  By Lemma~\ref{harnack:inequality},
\begin{align}
 w_e(x,t_n) &\geq \kappa ||w_e^+||_{L^\infty(D\times(t_n-3\theta,t_n-2\theta))}-e^{4m\theta}\sup_{\partial_P(U\times(t_n-4\theta,t_n))}w_e^-\nonumber\\
&> \kappa \alpha - e^{4m\theta}\sup_{\partial_P(U\times(t_n-4\theta,t_n))}w_e^-\label{ref1}
\end{align}
for $x\in D$ and $n\geq \cal N,$ where $m=\sup_{U\times(t_n-4\theta,t_n)} c^e$.  Since $z_e\geq 0$ in $\Sigma_1$ for all $z\in\omega(u)$, then
$||w_e^-(\cdot,t)||_{L^\infty(\Sigma)}\to 0$ as $t\to\infty.$  Therefore, passing  $n\to\infty$ in~\eqref{ref1} yields $\bar z_e(x) \geq \kappa\alpha>0$ for all $x\in D.$ Since $D$ can be arbitrarily big (with $\kappa > 0$ depending on $D$),  the result follows.
\end{proof}

\begin{lemma}\label{lemma4}
Fix any $e \in \S$ and suppose that $z_e>0$ in $\Sigma_1$ for all $z\in\omega(u).$ Then, there is $\varepsilon>0$ such that $z_{e'}>0$ in $\Sigma_1$ for all $z\in\omega(u)$ and $e'\in \S$ with $|e'-e|<\varepsilon.$
\end{lemma}

\begin{proof}
 Let $\rho_1$, $\delta$, and $G\subset\subset \Sigma_1$ be as in~\eqref{G}.  Since  $z_e>0$ for all $z\in\omega(u)$ and $G$ is a compact set, we have, by~\eqref{omega:unif}, that there are $\alpha,T>0$ such that $w_e(x,t)> 2\alpha>0$ $x\in G\times (T,\infty).$  By Lemma~\ref{regularity:lemma}, there is $\varepsilon>0$ such that, for $e'\in \S$ with $|e'-e|<\varepsilon$, one has
\begin{align}\label{nonzero}
 w_{e'}(x,t)> \alpha>0 \text{ for all } x\in G\times (T,\infty).
\end{align}
Let $\hat w_{e'}$ be as in Lemma~\ref{T:l}, then~\eqref{nonzero} implies that $\hat w_{e'}>0$ in $G\times(T,\infty),$ and by ~\eqref{equivalence},~\eqref{rho1}, and Lemma~\ref{weakmp}, 
\begin{align*}
 ||w^-_{e'}(x,t)||_{L^\infty(\Sigma)}&\leq ||\hat w^-_{e'}(x,t)||_{L^\infty(\Sigma \setminus G)}\leq e^{-\gamma (t-T)}||\hat w^-_{e'}(x,T)||_{L^\infty(\Sigma\setminus G)}\leq 4 e^{-\gamma (t-T)}|| w^-_{e'}(x,T)||_{L^\infty(\Sigma)}
\end{align*}
 for $t>T$, where we used that $\hat{w}^- \equiv 0$ and $c \leq -\gamma$ on $G$.  

Letting $t\to\infty$ we have that $z_{e'}\geq 0$ in $\Sigma_1$ and $z_{e'}\not\equiv 0$ by~\eqref{nonzero} for all $z\in\omega(u)$. The claim now follows by Lemma~\ref{lemma3}.
\end{proof}

The next Lemma is an adaptation of \cite[Lemma 3.8]{polacik:unbounded} to our setting and to the rotating plane method using the notation of Lemma~\ref{T:l}, we give a proof for completeness.

\begin{lemma}\label{lemma5}
Fix $e \in \S$ and  assume that $z_e>0$ in $\Sigma_1$ for all $z\in\omega(u),$ then there are $\varphi\in C^{2,1}(\Sigma_1\times(0,\infty))$ and $T(e)=T>0$ such that
\begin{enumerate}
\item if $G$ is as in~\eqref{G}, 
there is a bounded domain $D$ with $G\subset\subset D\subset\subset \Sigma_1$ such that $\varphi(x,t)<0$ in $\partial D\times (T,\infty)$ and $\varphi(x,t)>0$ in $G\times (T,\infty)$,
\item for all $T<s<t$ there exists some constant $C>0$ independent of $t$ and $s$ such that
\begin{align*}
        \frac{||\varphi(x,t)||_{L^\infty(D)}}{||\varphi(x,s)||_{L^\infty(G)}}\geq C e^{-\gamma(t-s)},
       \end{align*}
\item there is $\varepsilon>0$ independent of $e$ such that $\varphi_t<\Delta\varphi+\sum_{i=1}^N\hat b_i \varphi_{x_i}+\hat c^{e'}\varphi$ in $D\times(T,\infty)$ for all $e'\in \S$ with $|e-e'|<\varepsilon.$ Recall that $\hat b$ and $\hat c^e$ were defined in Lemma \ref{T:l}
\end{enumerate}
\end{lemma}

\begin{proof}
Let
\begin{align}
\delta\in(0,\gamma/2)\label{sol2}. 
\end{align}
By $(f_1)$ and Lemma \ref{regularity:lemma} there exists $\varepsilon>0$ such that 
\begin{align}\label{delta:coeff}
 || c^{e'}(\cdot,t)- c^e(\cdot,t)||_{L^\infty(\Sigma_1)}<\delta
\end{align}
for all $t>0$ and $e'\in \S$ with $|e-e'|<\varepsilon.$  Observe that for any $z\in \omega(u),$ one has  $z_e>0$ in $\Sigma_1$ and  $z_e(x)=0$ for $x\in \partial \Sigma_1$ 
or if $|x|\to\infty$. Hence, the compactness of $\omega(u)$ and the uniform convergence of $u(\cdot,t)$ to $\omega(u)$ as $t\to\infty$ (see (U2) 
implies that there exist $\alpha,t_0>0$
 such that
\begin{align}\label{intpositivity}
 w_e(x,t)>\alpha>0 \text{ in } G\times (t_0,\infty).
\end{align}
In addition, if we define 
\begin{align}
M&:=\sup\{| w_e(x,t)|:e\in \S, x\in \Sigma_1,t\in(0,\infty)\},\nonumber\\
m&:=\sup\{| c^e(x,t)|:e\in \S, x\in G,t\in(0,\infty)\},\nonumber\\
&0<s<\min\bigg\{\frac{\gamma \alpha}{2(m+\gamma)},\alpha\bigg\} \,, \label{sol1}
\end{align}
then there is bounded domain $D$ with $G\subset \subset D\subset \subset\Sigma_1$ 
such that if we increase $t_0$ if necessary, then 
\begin{align}\label{subsol:boundary}
  w_e(x,t)-s<0 \text{ in } \partial D\times (t_1,\infty) \,.
\end{align}
Let $\hat w_e$ be as in Lemma \ref{T:l}, then Lemma \ref{T:l} yields that \eqref{intpositivity} and \eqref{subsol:boundary} holds with $w_e$ replaced by $\hat w_e$. 
If we define $\varphi(x,t):= e^{-\gamma t}( \hat w_e(x,t)-s),$ by~\eqref{intpositivity}, \eqref{sol1}, and~\eqref{subsol:boundary},   claim 1. follows. Next, note that 
\begin{align*}
e^{\gamma t}||\varphi(x,t)||_{L^\infty(D)}\leq 4M + s\ \  \text{ and }\ \ e^{\gamma t}||\varphi(x,t)||_{L^\infty(G)}\geq \alpha -s
\end{align*}
for $t>T,$ which implies Claim 2.  Finally, for $(x,t)\in D\backslash G\times (T,\infty)$ by the definition of the coefficients $\hat c^e$ in Lemma~\ref{T:l},~\eqref{rho1},~\eqref{sol2},~\eqref{delta:coeff}, and~\eqref{intpositivity} yield
\begin{align*}
 \varphi_t-\Delta\varphi - \hat b_i\varphi_{x_i}-  \hat c^{e'}\varphi&=e^{-\gamma t} [( \hat c^e - \hat c^{e'}) \hat w_e - \gamma \hat w_e + s\gamma + s \hat c^{e'}]\leq e^{-\gamma t} [(\delta - \gamma) \hat w_e + s\gamma - s\gamma]<0.
\end{align*}

For $(x,t)\in G\times (T,\infty)$ by~\eqref{rho1},~\eqref{sol2},~\eqref{delta:coeff},~\eqref{intpositivity}, and~\eqref{sol1} we have
\begin{align*}
 \varphi_t-\Delta\varphi - \hat b_i \varphi_{x_i}-  \hat c^{e'}\varphi&=e^{-\gamma t} [( \hat c^e - \hat c^{e'})\hat w_e - \gamma \hat w_e + s\gamma + s \hat c^{e'}]\\
&\leq e^{-\gamma t} \bigg[\bigg(\delta - \frac{\gamma}{2}\bigg) \hat w_e -\frac{\gamma}{2} \hat w_e+ s (m+\gamma)\bigg] \leq e^{-\gamma t} \bigg[-\frac{\gamma \alpha}{2} + s (m+\gamma)\bigg]<0.
\end{align*}
This implies claim 3.
\end{proof}

\begin{proof}[Proof of Theorem~\ref{unbounded}]
We prove that the first alternative holds assuming that the second one does not hold. 
 Let $e\in \S$ as in $(U_0)$ and assume that $z_e\not\equiv 0$ for all $z\in\omega(u).$ Then, by Lemma~\ref{lemma3} and Lemma~\ref{lemma4}, there is $\varepsilon_0>0$ such that 
\begin{align}\label{open}
z_{e'}>0 \quad\text{ in $\Sigma_1$ for all } z\in\omega(u) \text{ and } e'\in \S \text{ with } |e'-e|<\varepsilon_0.
\end{align}

Now, let $P$ be any two dimensional subspace of $\mathbb R^N$ containing $e$. Without loss of generality, we may assume that $e=(1,0,\ldots,0)$ and $ P = \{ x=(x_1,x_2,0,\ldots,0) \mid x_1 , x_2 \in \mathbb R \}.$  Define 
\begin{align*}
e_\eta:=(\cos \eta,\sin \eta,0,\ldots,0),\quad 
w_{\eta}(x,t):=w_{e_\eta}(x,t),\quad 
z_{\eta}(x):=z_{e_\eta}(x)
\end{align*}
for $\theta \in \mathbb R$ and 
\begin{align*}
\Theta_1&:=\inf\{\eta<0 \::\: z_\phi > 0 \text{ in $\Sigma_1$ for all } z\in\omega(u), 0\leq \phi\leq \eta \},\\
\Theta_2&:=\sup\{\eta>0 \::\: z_\phi > 0 \text{ in $\Sigma_1$ for all } z\in\omega(u), 0\leq \phi\leq \eta\}.
\end{align*}
We note that $\Theta_1<0<\Theta_2$ by~\eqref{open}. We claim that $\Theta_2-\Theta_1 = \pi.$ By contradiction, assume that $\Theta_2-\Theta_1 < \pi;$ then, by continuity and Lemma~\ref{lemma3}, we can assume that there are $\bar z, \tilde z\in \omega(u)$ such that
\begin{align*}
 \tilde z_{\Theta_2}\equiv 0\quad\text{ and }\quad \bar z_{\Theta_1} \equiv 0\qquad \text{ in }\Sigma_1.
\end{align*}

Fix $\varepsilon_1>0$ as in Lemma~\ref{lemma5} part 3 and let 
\begin{align*}
0<\varepsilon<\min\bigg\{\pi-(\Theta_2-\Theta_1),\frac{\varepsilon_1}{2},\varepsilon_0\bigg\}.
\end{align*}
Fix $\theta\in(\Theta_1-\varepsilon,\Theta_1)$. By Lemma~\ref{lemma2:cor},
\begin{align}\label{omega1}
 \tilde z_{\theta} > 0\quad \text{ and } \quad\bar z_{\theta} < 0\qquad \text{ in } \Sigma_1.
 \end{align}
Let $\bar t_n,\tilde t_n\to\infty$ such that $\tilde t_n< \bar t_n$ for all $n\in\mathbb N$ and
 \begin{align}\label{omega2}
 w_{\theta}(x,\tilde t_n)\to \tilde z_{\theta}(x),\qquad w_{\theta}(x,\bar t_n)\to \bar z_{\theta}(x)\qquad \text{ as }n\to\infty.
  \end{align}

Now we conclude the proof arguing as in \cite[Lemma 3.7]{polacik:unbounded}, we include the details for completeness. Recall the change of variables detailed in Lemma~\ref{T:l}, where the coefficients $\hat b_i,$ $\hat c^e,$ the function $\hat w_e$ and the set $G\subset\Sigma$ were defined. 
For any $\eta \in [0, 2\pi)$ we also define $\hat w_\eta(x,t):=\hat w_{e_\eta}(x,t)$ for $(x,t)\in\Sigma_1\times (0,\infty).$

Since $z_\psi > 0$ for each $z \in \omega(u)$ and $\psi > \Theta_1$ sufficiently close to $\Theta_1$,
Lemma~\ref{lemma5} gives the existence of $C,T>0,$ a bounded domain $D$ with $G\subset\subset D\subset\subset \Sigma$ and a function $\varphi\in C^{2,1}(D\times (T,\infty)),$ that satisfies for $\vartheta\in(\Theta_1 - \varepsilon, \Theta_1)$ that 
\begin{equation}\label{symm1}
\begin{aligned}
\varphi_t-\Delta\varphi -\hat b_i \varphi_{x_i}-\hat c^{e_\vartheta}\varphi&<0 \quad\text{ in } D\times(T,\infty)\,, \nonumber\\
\varphi&<0 \quad\text{ on }\partial D\times(T,\infty),\nonumber\\
\varphi&>0 \quad\text{ in } G\times(T,\infty),\nonumber\\
\frac{||\varphi(\cdot,t)||_{L^\infty(D)}}{||\varphi(\cdot,s)||_{L^\infty(D)}}&\geq C e^{-\gamma(t-s)} \quad\text{ for } T<s<t,
\end{aligned}
\end{equation}
where we used that $|\vartheta - \psi| < 2\varepsilon \leq \varepsilon_1$.
By~\eqref{omega1}, \eqref{omega2}, and the definition of $\hat w_e$ there exists $\alpha>0$ such that
\begin{align*}
 \hat w_{\theta}(\cdot,\tilde t_n)>2\alpha>0\quad \text{ and }\quad
 \hat w_{\theta}(\cdot,\bar t_n)<0 \quad \text{ in } D
\end{align*}
for any sufficiently large $n$.  
Then, there is $T_n\in(\tilde t_n,\bar t_n)$ such that
\begin{align}
 \hat w_{\theta}(x,t)&>0 \text{ in } G\times(\tilde t_n,T_n),\\
 \hat w_{\theta}(x^*, T_n)&=0 \text{ for some } x^*\in \overline{G}.\label{contradiction}
\end{align}
From Lemma~\ref{weakmp} and~\eqref{rho1} follows
\begin{align}\label{maximum}
 ||\hat w_{\theta}^-(\cdot,t)||_{L^\infty(\Sigma)}\leq e^{-\gamma (t-\tilde t_n)}||\hat w_{\theta}^-(\cdot,\tilde t_n)||_{L^\infty(\Sigma)}  \ \ \ \text{ for } t\in (\tilde t_n,T_n).
\end{align}
By Lemma~\ref{regularity:lemma}, there is $\theta>0$ independent of $n$ such that 
 \begin{align*}
  ||\hat w_{\theta}(\cdot, t) - \hat w_{\theta}(\cdot,s)||_{L^\infty(\Sigma)}&<\alpha \text{ for } |t-s|<4\theta.
 \end{align*}
This means that $T_n-\tilde t_n >4\theta$ for all  sufficiently large $n$.  Let $n$ be 
sufficiently large such that $\tilde t_n>T.$ Since $\varphi$ satisfies \eqref{symm1} with $\vartheta = \theta$, 
by the comparison principle \cite[Proposition 52.6]{QS07},
\begin{align}\label{subsolution}
\hat w_\theta(x,t)>\alpha \frac{\varphi(x,t)}{||\varphi(\cdot,\tilde t_n)||_{L^\infty(D)}} \hspace{.5cm} \text{ for } (x,t)\in D\times(\tilde t_n, T_n).
\end{align}
And, by~\eqref{omega1},~\eqref{omega2},~\eqref{symm1},~\eqref{maximum},~\eqref{subsolution}, and Lemma~\ref{harnack:inequality}, there are $\kappa,C_1,C_2,m>0$ independent of $n$ such that
\begin{align*}
 \inf_{x\in G}\hat w_\theta(x,T_n)&>\kappa||\hat w_{\theta}^+||_{L^\infty(G\times(T_n-3\theta, T_n-2\theta))}- e^{4m\theta}\sup_{\partial_P (D\times(T_n-4\theta,T_n))} \hat w_{\theta}^-\\
&\geq \kappa \alpha C e^{-\gamma(T_n-\tilde t_n-2\theta)} - e^{4m\theta} e^{-\gamma(T_n-\tilde t_n - 4\theta)}||\hat w_{\theta}^-(\cdot,\tilde t_n)||_{L^\infty(\Sigma)}\\
&\geq e^{-\gamma(T_n-\tilde t_n)}[C_1 - C_2 ||\hat w_{\theta}^-(\cdot,\tilde t_n)||_{L^\infty(\Sigma)}]>0
\end{align*}
for $n$ sufficiently big, a contradiction to~\eqref{contradiction}.

Therefore $\Theta_2-\Theta_1 = \pi,$ which in particular implies that $z_{\Theta_1}\equiv z_{\Theta_2}\equiv 0$ in $\Sigma$ for all $z\in\omega(u).$ Since this can be done for all 2D planes $P$ that contain $e$ and the origin, Proposition~\ref{sec:char-foli-schw-1} yields the asymptotic symmetry of $u$. The strict monotonicity follows from the fact that $z_{\eta}>0$ for all $\eta\in(\Theta_1,\Theta_2).$ \end{proof}

We are ready to show Theorem~\ref{theorem:unbounded}. 

\begin{proof}[Proof of Theorem~\ref{theorem:unbounded}]
 By hypothesis $(U_0)'$ we have that $z_e\geq 0$ in $\Sigma_1$ for all $z\in\omega(u)$ and all $e\in U.$ If there exists $e\in U$ such that $z_e\not\equiv 0$ for all $z\in\omega(u)$ then Theorem~\ref{unbounded} yields the first alternative in the statement of the theorem.  On the other hand, if  for each $e\in U$ there is $z\in\omega(u)$ such that $z_e\equiv 0;$ then, Lemma~\ref{lemma2.5} implies that the second alternative in Theorem \ref{theorem:unbounded} holds. 
\end{proof}

\begin{proof}[Proof of Corollaries~\ref{corollary:elliptic} and~\ref{corollary:periodic}]
Corollaries~\ref{corollary:elliptic} and~\ref{corollary:periodic} follow directly from Theorem~\ref{unbounded}, since in these cases, the strict initial reflectional inequality and the maximum principle discard the second alternative in Theorem~\ref{unbounded}.
\end{proof}

\begin{proof}[Proof of Theorem \ref{thm2}]
  Let $f(t,u):=a(t)|u|^{p-1}u-b(t)u$. Then $f_u(t,u):=p\, a(t)|u|^{p-1}-b(t)$ and $(f_1)$ holds because
 \begin{align*}
  \lim_{u\to v}\sup_{t>0}|f_u(t,u)-f_u(t,v)|
  \leq p\|a\|_\infty\lim_{u\to v}\Big||u|^{p-1}-|v|^{p-1}\Big|=0
 \end{align*}
and, for every $K>0$,
\begin{align*}
 \sup_{r\in J,t>0,s\in[-K,K]}|f_u(t, s)|
 \leq p\max\{\|a\|_\infty,\|b\|_\infty\}|K^{p-1}+1|<\infty.
\end{align*}
Moreover, if
$\eta>0$ is as in \eqref{eta:intro} and
$\eps:=\left(\frac{\eta}{2p\max\{\|a\|_\infty, 1\}}\right)^\frac{1}{p-1}$, then $\sup\limits_{u\in(-\eps,\eps)}|p\, a(t)|u|^{p-1}|\leq\frac{\eta}{2}$ and
\begin{align*}
 f_u(t,u)\leq |p\, a(t)|u|^{p-1}| - \eta=-\frac{\eta}{2}\quad \text{ for all }t>0 \text{ and }u\in(-\eps,\eps),
\end{align*}
which implies that $(f_2)$ is satisfied.  Finally, by \eqref{u0h}, condition $(U_0)'$ holds.  Then, by Theorem \ref{theorem:unbounded}, either there is $z\in\omega(u)$ radially symmetric with respect to the origin, or $u$ is asymptotically strictly foliated Schwarz symmetric. 

It remains to prove that
\begin{itemize}
 \item [(a)] if $z\in\omega(u)$ is radially symmetric, then $z\equiv 0$. 
 \item [(b)] if $0\not\in \omega(u)$, then there is $q\in\R^N$ such that, for all $z\in\omega(u)$, $z$ is radially symmetric with respect to $q$ and all elements in $\omega(u)$ are either all negative or all positive.
\end{itemize}
 
 \textbf{Proof of (a):} For a contradiction, assume there is $z\in\omega(u)$ radially symmetric and $x_0\in\R^N$ such that $z(x_0)\neq 0$. Using the radiality and continuity of $z$ we may  without loss of generality
assume that $x_0=re_1=(r,0,\ldots,0)$ for some $r>0$ and $z(re_1)<0$. For $\lambda\in(-1,1)$, denote $\R^N_\lambda:=\{x\in\R^N\::\: x_1>\lambda\}$ and $u_\lambda(x,t):=u(x,t)-u(2\lambda e_1-x,t)$ for $x\in\R^N_\lambda$ and $t\geq 0$.  By \eqref{u0h}, $u_\lambda (x,0) \geq 0$ and $u_\lambda (x,0) \neq 0$ for any $\lambda \in(-1, 1)$, and therefore, by 
the maximum principle, $u_\lambda>0$ in $\R^N_\lambda\times(0,\infty)$. Hence, $z(x)\geq z(2\lambda e_1-x)$ for all $x\in\R^N_\lambda$ and $\lambda\in(-1,1)$. 
Since $z$ is radial and in particular even, for any fixed
 $0<\lambda<\min\{r,1\}$ one has $z_{-\lambda} \geq 0$ in $\R^N_{-\lambda}$, and then
\begin{align}\label{ab}
 0&>z(re_1)
 \geq z(-2\lambda e_1-re_1)=z((2\lambda+r)e_1).
 %\geq z(-2\lambda e_1-(2\lambda+r)e_1)=z((4\lambda+r)e_1).
\end{align}
Iterating this procedure, we obtain that $0>z(re_1)\geq z((2k\lambda+r)e_1)$ for all $k\in\mathbb N$. Since $|(2k\lambda+r)e_1|\to\infty$ as $k\to\infty$ we obtain 
a contradiction to the uniform decay assumption  \eqref{udintro}. As a consequence, $z\equiv 0$ in $\R^N$.

\medskip

\textbf{Proof of (b):} Assume that $0\not\in \omega(u)$. Then, as shown above, there is $p_0\in\S$ such that, for all $z\in\omega(u)$, $z$ is strictly foliated Schwarz symmetric with respect to $p_0$. Let $e_1=(1,0,\ldots,0)$, $e_2=(0,1,0,\ldots,0)$, $\ldots$, $e_N=(0,0,\ldots,1)$ denote a basis for $\R^N.$  For $\alpha\in\R$ and $i\neq 1$, consider 
$u_{\alpha,i}(x):=u(x+\alpha e_i)$ and note that $u_{\alpha,i}$ solves the same equation as $u$ (which is translationally invariant) and also satisfies $(U_0)'$. Then, arguing as before, there is $p_{\alpha,i}\in \S$ such that, for all $z\in\omega(u_{\alpha,i})$, $z$ is foliated Schwarz symmetric with respect to $p_{\alpha,i}$; namely,
every $z\in\omega(u)$ is strictly foliated Schwarz symmetric with respect to the axes $P_{i,\alpha}:=\R p_{\alpha,i}+\alpha e_i$.  

Fix $\bar z\in\omega(u)$ and assume without loss of generality that $\bar z^+\not\equiv 0$. Since $\bar z$ is strictly monotone with respect to the polar angle and it decays uniformly to zero at infinity, then the maximum of $z$ in $\R^N$ must lie in an intersection between $P_{i,\alpha}$ and $P_0$, $q:=P_{i,\alpha}\cap P_0$.  By varying $\alpha\in\R$, we can obtain an irrational angle between $P_{i,\alpha}$ and $P_0$, which then easily implies that $\bar z$ must be a radially symmetric function with respect to $q$. But the axis $P_0$ and $P_{i,\alpha}$ are symmetry axes for \emph{all} elements $z\in\omega(u)$, and therefore $\omega(u)$ can only have radially symmetric elements with respect to $q$.

Finally, for a contradiction, assume that there is a sign-changing $z\in\omega(u)$. Assume first that $q\in \{x_1\geq 0\}$.  We argue as in \eqref{ab} on the line $q+\R e_1$.  Indeed, since $z$ changes sign and $z$ is radially symmetric and continuous, there is $r>0$ such that $z(q+re_1)<0$. But then, for $0<\lambda<\min\{r,1\}$,
\begin{align*}
 0&>z(q+re_1)
 \geq z(-2\lambda e_1-q-re_1)=z(q+(2\lambda+r)e_1).
\end{align*}
Iterating this procedure, we obtain a contradiction to the uniform decay assumption  \eqref{udintro} as before. In fact, we can deduce that if $q\in \{x_1\geq 0\}$, then necessarily $z\geq 0$ in $\R^N$ for all $z\in\omega(u)$, and the strict positivity follows from Lemma \ref{harnack:inequality} using that $0\not\in\omega(u)$.  Similarly, if $q\in \{x_1 \leq  0\}$, then, since $z$ changes sign and $z$ is radially symmetric and continuous, $z(q-re_1)>0$ for some $r>0$ and, for $0<\lambda<\min\{r,1\}$,
\begin{align*}
 0&<z(q-re_1)
 \leq z(2\lambda e_1-q+re_1)=z(q-(2\lambda+r)e_1).
\end{align*}
Iterating this procedure, we obtain again a contradiction to \eqref{udintro} and, as before, we conclude that if $q\in \{x_1\leq 0\}$, then necessarily $z\leq 0$ in $\R^N$ for all $z\in\omega(u)$ and the strict negativity follows from Lemma \ref{harnack:inequality} applied to $-u$. Finally, observe that if $q\in\{x_1=0\}$, then necessarily $z\equiv 0$ for all $z\in\omega(u)$, which cannot happen since we assumed that $0\not\in\omega(u)$. Therefore, $q\not\in\{x_1=0\}$ and this ends the proof.
\end{proof}

\section{Uniform Decay Assumption}\label{hypothesis:discussion}

\begin{proof}[Proof of Lemma~\ref{UD1}]
 The proof follows closely the ideas of \cite[Corollary 1.2]{busca}. We first show that there is some $R_0>0$ such that
\begin{align}\label{gradient:decreasing}
 \nabla u(x,t)\cdot e <0 \hspace{1cm} \text{ for all } x\in\mathbb R^N, \ x\cdot e>R_0, \ t>0, \ e\in \S.
\end{align}
To this end, we use a moving plane method. Let $R>0$ as in~\eqref{weak:monotonicity} and $e\in \S,$ without loss of generality we may assume that $e=e_1=(1,0,\ldots,0).$ For $\lambda>R$ and $(x_1,x')=x\in\mathbb R^N$ define $x\mapsto x^\lambda=(2\lambda - x_1,x')$ the reflection in the hyperplane $H_\lambda=\{x\in\mathbb R^N : x_1=\lambda\}$ and let $V_\lambda u(x,t):=u(x^\lambda,t)-u(x,t)$ for $x\in\Sigma_\lambda:=\{x\in\mathbb R^N : x_1>\lambda\}$ and $t>0.$  Then, by~\eqref{weak:monotonicity}, we have that
\begin{align*}
 (V_\lambda u(x,t))_t-\Delta V_\lambda u(x,t)&=f(t,|x^\lambda|,u^\lambda)-f(t,|x|,u)\\&\geq f(t,|x|,u^\lambda)-f(t,|x|,u)=c^\lambda(x,t)V_\lambda u(x,t)
\end{align*}
for all $(x,t)\in \Sigma_\lambda\times(0,\infty),$ where $c^\lambda(x,t):=\int_0^1 \partial_u f(t,|x|,su(x^\lambda,t)+(1-s)u(x,t))ds.$ Note that $c^\lambda\in L^\infty(\Sigma_\lambda \times (0, \infty)).$  Let $R_0\geq R$ be such that $\operatorname{supp}(u_0)\subset B_{R_0}(0),$ then, for $\lambda > R_0$ we have that $V_\lambda u(x,0)\geq 0$ in $\Sigma_\lambda$ and $V_\lambda u(x,0)\not\equiv 0.$  Then, since $u$ is globally bounded, we can apply the parabolic maximum principle to show that $V_\lambda u(x,t)>0$ in $\Sigma_\lambda\times(0,\infty)$ for all $\lambda>R_0$ and $V_\lambda u(x,t)\equiv 0$ in $H_\lambda\times(0,\infty).$  Then, by the Hopf's Lemma we have that
$-2 \nabla u(x,t)\cdot e= -2\partial_e u(x,t)=\partial_e V_\lambda u(x,t) >0$ on $H_\lambda\times(0,\infty)$ and~\eqref{gradient:decreasing} follows. Next, to prove the uniform decay property. We proceed by contradiction. Assume there is some sequence $\{(x_n,t_n)\}_{n=1}^\infty\subset \mathbb R^N\times (0,\infty)$ with $|x_n|,t_n\to \infty$ as $n\to\infty$ such that
\begin{align}
 u(x_n,t_n)\geq k
\end{align}
for some constant $k>0.$  Passing to a subsequence, we may assume (using rotations if necessary) that $x_k/|x_k|$ tends to the unitary vector $e_1.$  By~\eqref{gradient:decreasing} we have that $\partial_{x_1}u(x,t)<0$ for any $t>0$ and any $x=(x_1,x')\in\mathbb R^N$ with $x_1$ big enough.  Then, the assumed $L^\infty$ bound and standard parabolic 
regularity estimates imply that $|\nabla u|$ is uniformly bounded.  Then, we can find $\varepsilon>0$ such that $u(x,t_n)\geq \frac{k}{2}$ for all $x\in B_\varepsilon(x_n).$  Since $|x_n|\to \infty$ and $\partial_{x_1}u(x,t)<0$ we obtain a contradiction to the assumed $L^q$-bound.
\end{proof}

\begin{proof}[Proof of Lemma~\ref{UD2}]
 We proceed by comparison. Consider the solution of the problem
\begin{equation*}
\begin{aligned}
v_t -\Delta v &= f(t,|x|,v)\quad \text{ in }\mathbb R^N \times (0,\infty),\qquad v(x,0) = u^+_0(x)\quad \text{ for }x\in \R^N;
\end{aligned}
\end{equation*}
then, by Lemma~\ref{UD1}, $\lim_{|x|\to\infty}\sup_{t>0} v(x,t)=0.$  On the other hand, let $w(x,t):=v(x,t)-u(x,t),$ then $w$ satisfies the linearization
\begin{equation*}
w_t -\Delta w = \bigg[\int_0^1 \partial_u f(t,|x|,sv+(1-s)u)ds\bigg] w\quad \text{ in }\mathbb R^N \times (0,\infty)
\end{equation*}
and $w(x,0) \geq 0$ for $x\in \mathbb R^N$.  Since the  coefficient of $w$ is bounded, the parabolic maximum principle implies that $v\geq u$ in $\mathbb R^N\times (0,\infty).$  Then, $\lim_{|x|\to\infty}\sup_{t>0} u(x,t)\leq 0.$ Now, using the solution of 
\begin{equation*}
\begin{aligned}
v_t -\Delta v &= f(t,|x|,v)\quad \text{ in }\mathbb R^N \times (0,\infty),\qquad 
v(x,0) = u^-_0(x)\quad \text{ for }x\in \R^N,
\end{aligned}
\end{equation*}
and repeating the argument we obtain that $\lim_{|x|\to\infty}\sup_{t>0} u(x,t)\geq 0$ and the claim follows.
\end{proof}

\section{Examples}\label{ex:sec}

\subsection{Example 1: An $\omega$-limit set with only strictly foliated Schwarz symmetric elements}\label{e1}

Here we exhibit a problem which does not have any radially symmetric\footnote{with respect to the origin or with any other point $x_0\in\R^N$.} element in the $\omega$-limit set of the solution $\omega(u)$, but that satisfies the assumptions of Theorem~\ref{theorem:unbounded}, and therefore all elements in  $\omega(u)$ are foliated Schwarz symmetric with respect to the same axis. 

First, fix any ball $B \subset \R^N_+:=\{x\in \R^N: x_1 > 0\}$ and let $\lambda_1(B)$ denote the first eigenvalue of the Dirichlet Laplacian in $B$.  Define $g(v) =  v^p-\lambda v$,  where $1 < p < (N + 2)/(N - 2)$ and $0<\lambda<\lambda_1(B)$.  Then, by standard variational arguments (see e.g. \cite[Theorem 6.2]{QS07}) there exists a positive solution $\zeta$ of the problem
\begin{equation}\label{zeq}
-\Delta \zeta = g(\zeta) \quad \textrm{in } B, \qquad \zeta = 0 \quad  \textrm{on } \partial B \,.
\end{equation}   
Let $M = 4 \sup_B \zeta$ and define $f$ such that $f(|x|, v) = g(v)$ for $x \in B$ and $0 \leq v < M/2$. 
For $v > M/2$ and $x \in \R^N \setminus B$,  $f:[0,\infty)\times \R$ is extended as a smooth function such that it satisfies 
 $(f_0)$ and $(f_1)$.  We assume also that
\begin{itemize}
\item[$(f_3)$] $f$ is an odd function in $u$, that is, $f(\cdot,u)=-f(\cdot,-u)$ for all $u\in\R$,
\item[$(f_4)$] there is $M^*>M$ such that
\begin{equation}
 uf(r,u) < 0 \qquad \text{ for any } r > 0 \text{ and } |u|\geq M^*,
\end{equation}
 \item[$(f_5)$] there is $R^* > 1$ and $\Lambda > 0$ such that
 \begin{equation}
 f(r, 0) = 0 \qquad \textrm{and} \qquad 
 f_u(r, u) < -\Lambda \qquad \text{ for any } |u| \leq M^*\text{ and }r > R^* \,.
 \end{equation}
\end{itemize}

Let $u_0\in C(\R^N)$ be an odd function with respect to $x_1$, that is, $u_0(x_1,x')=-u_0(-x_1,x')$ for all $x=(x_1,x')\in\R^N$, and assume that 
\begin{align}\label{u0s}
\operatorname{supp}(u_0)=A_1\cup A_2, 
\end{align}
where $A_1$ and $A_2$ are two disjoint compact sets such that $B \subset A_1$,   
\begin{align}\label{cu0}
A_1&\subset\subset \{x\in\R^N\::\: |x| < R^*, x_1>0,\ x_i>0\ i=2,\ldots,N\},
\end{align}
and $A_2$ is symmetric to $A_1$ with respect to the hyperplane $\{x \in \R^N : x_1 = 0\}$. Finally, assume that 
\begin{align}\label{u0z}
u_0 \geq \zeta\quad \text{ in $B$,}\qquad |u_0|<M^*\quad \text{ in $\R^N$,}
\end{align}
 and 
 $u_0$ is positive in $A_1$ and negative in $A_2$.  
 
 \medskip
 
 Let us first derive a priori estimates for classical solutions $u$ of 
\begin{align}\label{weq}
 u_t - \Delta u = f(|x|,u)\quad \text{ in }\R^N \times (0, \infty),\qquad u(x,0)=u_0(x)\quad \text{ for }x\in\R^N \,.
\end{align}
Note that, by $(f_4)$,
\begin{align*}
 (u-M^*)_t-\Delta (u-M^*) - \frac{f(|x|,u)-f(|x|,M^*)}{u-M^*}(u-M^*)=f(|x|,M^*)<0 \quad \text{ in }\R^N\times(0,\infty),
\end{align*}
where, by $(f_4)$, $\sup_{x\in\R^N,t>0}\frac{f(|x|,u)-f(|x|,M^*)}{u-M^*}<\infty$.  Then, by \eqref{u0z} and the maximum principle, $u \leq M^*$ in $\R^N\times(0,\infty)$.  Arguing similarly with $M^*+u$, we obtain that
 \begin{align}\label{u2s}
 |u| \leq M^*\quad \text{ in $\R^N\times(0,\infty)$.}
 \end{align}
Then, by standard arguments (see, e.g., \cite[Proposition 7.3.1]{Lu95}), there exists a global bounded, smooth solution $u$ of \eqref{weq}. 
 
Observe that $\widehat w(x,t):= u(x,t)+u(x^{e_1},t)$ solves
\begin{align*}
 \widehat w_t-\Delta \widehat w = \widehat c(x, t) \widehat w\quad \text{ in }\R^N\times(0,\infty),
\end{align*}
where $\widehat c(x,t):=\frac{f(u(x,t))+f(u(x^{e_1},t))}{\widehat w(x,t)}$ if $\widehat w(x,t)\neq 0$ and  $\widehat c(x,t)=0$ if $\widehat w(x,t)=0$. Since $\widehat w(x,0)=u_0(x) + u_0(x^{e_1}) \equiv 0$ in $\R^N_+ := \{x \in \R^N : x_1 > 0\}$ we have that $\widehat w\equiv 0$ in $\R^N \times [0, \infty)$.  In particular, it means that $u(\cdot, t)$ is odd in $x_1$ for all $t > 0$. Therefore, $u(0, x', t) = 0$ for all $t>0$ and $x'\in\R^{N-1}$.  Since $u_0$ is nonnegative and nontrivial in $\R^N_+$, the maximum principle implies that 
\begin{align}\label{upRp}
 u > 0\quad \text{ in $\R_+^N\times(0,\infty)$. }
\end{align}

Next, let
 $s(x) := M^*e^{- \theta (|x| - R^*)}$ with $\theta^2<\Lambda$; then,
 \begin{equation}
 s_t - \Delta s = \left(\theta\frac{N - 1}{|x|}-\theta^2\right)s \qquad \text{for $|x| \geq R^*$}.
 \end{equation}
Therefore, by $(f_5)$, $w := s - u$ satisfies
\begin{align*}
w_t - \Delta w -\int_0^1f_u(|x|,\theta u)\, d\theta\, w
&=  \left(\theta\frac{N - 1}{|x|}-\theta^2-\int_0^1f_u(|x|,su)\, ds\right)\, s
\geq  \left(\Lambda-\theta^2\right)\, s>0
\end{align*}
in $\R^N\setminus B_{R^*}$.  Since $w =  M^* - u \geq 0$ for $x \in \partial B_{R^*}$ and $w(x,0) = s(x)-u_0(x)=s(x)$ for $x \in \R^N \setminus B_{R^*}$, the maximum principle yields that $w>0$ on  $\R^N \setminus B_{R^*} \times (0, \infty)$. Thus, 
\begin{align}\label{U1s}
0 \leq u \leq  M^* e^{- \theta (|x| - R)}\quad \text{ in $\R^N \setminus B_{R^*} \times (0, \infty)$}.
\end{align}

Since $f$ coincides with $g$ on the range of $\zeta$ and $u>0$ on $\partial B\times(0,\infty) \subset \R^N_+\times(0,\infty)$, we have that
\begin{align*}
(u-\zeta)_t-\Delta(u-\zeta)-\frac{g(u)-g(\zeta)}{u-\zeta}(u-\zeta)=0 \quad \text{ in }B\times(0,\infty).
\end{align*}
Moreover, by \eqref{upRp}, $u-\zeta>0$ on $\partial B\times(0,\infty)$ and, by \eqref{u0z}, $u(\cdot,0)-\zeta\geq 0$ in $B$.  Then, by the maximum principle, $u \geq \zeta$ in $B\times(0,\infty)$; in particular,  $0\notin \omega (u)$.  Furthermore, since $u$ is odd in $x_1$, any $z \in \omega(u)$ is odd in $x_1$ as well, and consequently $z$ is not radially symmetric. 

Finally, note that all the assumptions of Theorem \ref{theorem:unbounded} are satisfied; indeed, $(f0)$ and $(f1)$ hold by construction, $(f_2)$ follows from $(f_5)$, $(U_0)'$ from \eqref{u0s}, $(U_1)$ from \eqref{U1s}, and $(U_2)$ from \eqref{u2s}. Therefore, we obtain that all the elements in $\omega(u)$ are strictly foliated Schwarz symmetric.

\subsection{Example 2: An $\omega$-limit set with a strictly foliated Schwarz symmetric function and a radially symmetric element}\label{e2}

In this example we show that the presence of a nonzero radially symmetric element in $\omega(u)$ does \emph{not} imply that all the elements in $\omega(u)$ are radially symmetric. Although we work in an abstract setting  we provide concrete examples. 

Let $D$ be a smooth radial domain, bounded or unbounded, in $\R^N$ and assume that a linear operator $\mathcal{L}$ acting on  smooth functions on $D$ has eigenvalues 
$0 < \lambda_1 < \lambda_2$ corresponding to eigenfunctions $\varphi_1$ and $\varphi_2$ such that $\varphi_1$ is foliated Schwarz symmetric and $\varphi_2$ is radially symmetric. Such choice is possible since there exist radial eigenfunctions corresponding to arbitrarily large eigenvalues. Note that we require the existence of a radial eigenfunction corresponding to a larger eigenvalue than the foliated Schwarz symmetric one. 
To satisfy our assumptions we allow $\mathcal{L}$ to depend only on $|x|$.
Let $u(x,t):= \alpha(t) \varphi_1(x) + \beta(t) \varphi_2(x)$ for $x\in D$ and $t>0$ and, for fixed $\mu \in (\lambda_1, \lambda_2)$ and define
\begin{equation}
f(t, x, u) =  \mu \zeta(t) u  + \psi(t) \varphi_2(x)\qquad x\in D,\ t>0,
\end{equation}
where  $\alpha, \beta, \zeta, \psi : (0, \infty) \to \R$ are chosen below. Note that, since $\varphi_2$ is radially symmetric, $f$ depends on $x$
only through $|x|$. 

First we require that $\alpha(0) = 1$ and $\beta(0) = 0$; therefore, $u_0(x) = \phi_1(x)$ satisfies $(U_0)$.
Also we assume that 
$u$ is a solution of 
\begin{align*}
u_t + \mathcal{L} u = f(t, x, u) \quad \text{ in } D \times (0, \infty), \qquad u(x, t) = 0 \quad \text{ for } x \in \partial D \,,  
\end{align*}
where $u = 0$ on $\partial D$ is interpreted as $u (x) \to 0$ when $|x| \to \infty$ if $D$ is unbounded.

Since the operator is linear and $\varphi_i$ is an eigenfunction we immediately have that $u$ is the desired solution if
\begin{align}
\alpha' + \lambda_1  \alpha = \mu \zeta \alpha, \qquad \beta' + \lambda_2 \beta  =  \mu \zeta \beta  + \psi, \qquad \alpha(0) = 1, \quad \beta(0) = 0 \,. 
\end{align}
Consequently, by the variation of parameters formula, for any $t > t_0 \geq 0$ we have that
\begin{align}\label{dfa}
\alpha(t) &= \alpha(t_0) \exp\left( \int_{t_0}^t \mu \zeta(s) - \lambda_1 \, ds \right) \,, \\ \label{dfb}
\beta(t) &= \int_{t_0}^t \psi(s) \exp\left( \int_s^t  \mu \zeta(r) - \lambda_2\, dr \right) ds + \beta(t_0)  \exp\left( \int_{t_0}^t  \mu\zeta(s) - \lambda_2\, ds \right)  \,.
\end{align}
We finish the argument by constructing sequences $(T_k)_{k \in \mathbb{N}}$, $(\bar{T}_k)_{k \in \mathbb{N}}$ such that 
\begin{equation}\label{efar}
\alpha(T_k) = 1, \quad |\beta(T_k)| \leq \frac{1}{2^k}, \qquad 0 < |\alpha(\bar{T}_k)| \leq \frac{1}{2^k}, \quad 
|\beta(\bar{T}_k) - 1| \leq  \frac{1}{2^k} \,.
\end{equation}

Indeed, then we obtain
\begin{equation}
\lim_{k \to \infty} u(\cdot, T_k) = \varphi_1 \quad \lim_{k \to \infty} u(\cdot, \bar{T}_k) = \varphi_2
\end{equation}
and the statement follows. 

To prove~\eqref{efar}, we proceed by induction and in addition we show that $\zeta(T_n) = 1$ and $\psi(T_n) = 0$ for every $n$. Set $T_1 = 0$
and
assume that we already constructed $T_1 < \bar{T}_1 < T_2  <  \cdots  < \bar{T}_{n-1} < T_n$ 
with the desired properties \eqref{efar}. 

First introduce a short transition period (of order) to shift the values of $\zeta$ and $\psi$ to 0 and $\lambda_2$ respectively. We use linear functions, but 
if the smoothness at the endpoints is required, the transition can me made smooth at the cost of less explicit expressions. Define
\begin{equation}
\zeta(t) = 1 + T_n - t, \qquad \psi(t) = \lambda_2 (t - T_n) \qquad t \in [T_n, T_n + 1) 
\end{equation}
and, by \eqref{dfa} and \eqref{dfb},
\begin{equation}
\alpha(T_n + 1) = \alpha(T_n) e^{\frac{\mu}{2} - \lambda_1} \,, \quad 
\beta(T_n + 1) = c_1
+ \beta(T_n) e^{\frac{\mu}{2} - \lambda_2} \,,
\end{equation}
where $c_1$ is a universal constant independent of $n$. 
Then, for fixed $A_1$ specified below, define 
\begin{equation}
\zeta(t) = 0, \qquad \psi(t) = \lambda_2, \qquad t \in [T_n + 1, T_n + A_1 + 1) \,,
\end{equation}
and set $\bar{T}_n = T_n + 1 + A_1$.  Consequently, by \eqref{dfa} and \eqref{dfb}
\begin{equation}
\alpha(\bar{T}_n) = \alpha(T_n + 1) e^{-\lambda_1 A_1}, \qquad 
\beta(\bar{T}_n) = 1 - e^{-\lambda_2 A_1} + \beta(T_n + 1) e^{-\lambda_2 A_1}\,.
\end{equation}
Thus, for any sufficiently large $A_1$,  $\alpha(\bar{T}_n)$ and $\beta(\bar{T}_n)$ satisfy~\eqref{efar}. 
Next, we introduce the second transition period  that connects $(\zeta (\bar{T}_n), \psi (\bar{T}_n)) = (0, \lambda_2)$ 
to $(1, 0)$. Set 
\begin{equation}
\zeta(t) = t - \bar{T}_n, \qquad \psi(t) = \lambda_2 (\bar{T}_n + 1 - t), \qquad t \in [\bar{T}_n, \bar{T}_n + 1) \,,
\end{equation}
and therefore 
\begin{equation}
\alpha(\bar{T}_n + 1) = \alpha(\bar{T}_n)  e^{\mu/2 - \lambda_1}, \qquad \beta(\bar{T}_n + 1) = 
\beta(\bar{T}_n)c_2 + c_3\,,
\end{equation}
where $c_2, c_3$ are universal constants independent of $n$ or $A_1$ (which can be evaluated explicitly).  Then, set
\begin{equation}
\zeta(t) = 1, \qquad \psi(t) = 0, \qquad t \in [ \bar{T}_n + 1, \bar{T}_n + 1 + A_2) \,,
\end{equation}
where $A_2$, depending on $A_1$ is  specified below.  Denote $T_{n + 1} = \bar{T}_n + 1 + A_2$ and we obtain
\begin{align}
\alpha(T_{n + 1}) &= \alpha(T_n)e^{-\lambda_1A_1} e^{\mu - 2\lambda_1} e^{(\mu - \lambda_1)A_2}, \\
 \beta(T_{n + 1}) &= \beta(\bar{T}_n + 1)  e^{(\mu - \lambda_2) A_2} \,.
\end{align}
Since $\mu >  \lambda_1$, we can set 
\begin{equation}
A_2 = \frac{\lambda_1 A_1 + 2 \lambda_1 - \mu}{\mu - \lambda_1} > 0
\end{equation}
to obtain $\alpha(T_{n + 1}) = \alpha(T_n) = 1$. Also, by making $A_1$ larger if necessary, $\mu < \lambda_2$ yields that 
$\beta(T_{n + 1})$ satisfies~\eqref{efar}. Clearly, $\zeta(T_{n + 1}) = 1$ and $\psi(T_{n + 1}) = 0$ and the proof is complete. 

\begin{exam}
Let $B \subset \R^N$ with $N \geq 2$ be a ball and let $\mathcal{L} = - \Delta$. It is standard to show that such $\mathcal{L}$ has an increasing sequence of eigenvalues diverging to infinity corresponding to the radial eigenfunctions. Furthermore, the first non-radial eigenfunction is foliated Schwarz symmetric. Hence, the assumptions are satisfied and we can find $f$ such that the $\omega$-limit set contains both radial and foliated Schwarz symmetric functions. 
\end{exam}

\begin{exam}
If $D = \R^N$ with $N \geq 2$, we can set $\mathcal{L} = - \Delta + V(|x|^2)$, where $V \geq 0$ and $V(r) \to \infty$ as $r \to \infty$. 
If we denote $L^2(\R^N, V)$ the $L^2$ space with weight $V$, 
then, due to compactness of the embedding $L^2(\R^N) \hookrightarrow H^1(\R^N) \cap L^2(\R^N, V)$ and the spectral theorem, 
one obtains that $\mathcal{L}$ has only discrete spectrum and the eigenvectors span the whole space. 

Also, the method of  separation of variables implies that the 
eigenvalue problem 
\begin{equation}
\mathcal{L} u = \lambda u
\end{equation}
with $u(r, \theta) = R(r) \Theta(\theta)$, $r \in [0, \infty)$, $\theta \in \S$ can be written as 
\begin{equation}
-\Delta_r R + V(r^2) R = \lambda_1 R, \qquad R'(0) = 0
\end{equation}
and 
\begin{equation}
-\Delta_{\theta} \Theta = \lambda_2 \Theta \,,
\end{equation}
where $\Delta_r$ is radial Laplacian and $\Delta_\theta$ is Laplace-Beltrami operator on $\S$. 

It is easy to show that such $\mathcal{L}$ has an increasing sequence of 
eigenvalues diverging to infinity corresponding to the radial eigenfunctions. Furthermore, the first non-radial eigenfunction is foliated Schwarz symmetric. 
Hence, the assumptions are satisfied and we can find $f$ such that $\omega$-limit set contains both radial and foliated Schwarz symmetric functions. 
\end{exam}

\subsection*{Acknowledgments}
J. F\" oldes is partly supported by the National Science Foundation under the grant NSF-DMS-1816408.  A. Saldaña is supported by UNAM-DGAPA-PAPIIT grant IA101721, Mexico.

\bibliographystyle{plain}

\begin{thebibliography}{10}

\bibitem{bere}
H.~Berestycki and L.~Nirenberg.
\newblock On the method of moving planes and the sliding method.
\newblock {\em Bol. Soc. Brasil. Mat. (N.S.)}, 22(1):1--37, 1991.

\bibitem{brock}
F.~Brock.
\newblock Symmetry and monotonicity of solutions to some variational problems
  in cylinders and annuli.
\newblock {\em Electron. J. Differential Equations}, pages No. 108, 20, 2003.

\bibitem{busca}
J.~Busca, M.~A. Jendoubi, and P.~Polacik.
\newblock Convergence to equilibrium for semilinear parabolic problems in
  {$\mathbb R^N$}.
\newblock {\em Comm. Partial Differential Equations}, 27(9-10):1793--1814,
  2002.

\bibitem{ch82}
Ll.~G. Chambers.
\newblock An upper bound for the first zero of {B}essel functions.
\newblock {\em Math. Comp.}, 38(158):589--591, 1982.

\bibitem{f5}
J.~F\"oldes.
\newblock On symmetry properties of parabolic equations in bounded domains.
\newblock {\em J. Differential Equations}, 250(12):4236--4261, 2011.

\bibitem{f6}
J.~F\"oldes.
\newblock Symmetry of positive solutions of asymptotically symmetric parabolic
  problems on {$\mathbb R^N$}.
\newblock {\em J. Dynam. Differential Equations}, 23(1):45--69, 2011.

\bibitem{f2}
J.~F\"oldes.
\newblock On {S}errin's symmetry result in nonsmooth domains and its
  applications.
\newblock {\em Adv. Differential Equations}, 18(5-6):523--548, 2013.

\bibitem{fp}
J.~F\"oldes and P.~Polacik.
\newblock On cooperative parabolic systems: {H}arnack inequalities and
  asymptotic symmetry.
\newblock {\em Discrete Contin. Dyn. Syst.}, 25(1):133--157, 2009.

\bibitem{f4}
J.~F\"oldes and P.~Polacik.
\newblock Convergence to a steady state for asymptotically autonomous
  semilinear heat equations on {$\mathbb R^N$}.
\newblock {\em J. Differential Equations}, 251(7):1903--1922, 2011.

\bibitem{f3}
J.~F\"oldes and P.~Polacik.
\newblock On asymptotically symmetric parabolic equations.
\newblock {\em Netw. Heterog. Media}, 7(4):673--689, 2012.

\bibitem{f1}
J.~F\"oldes and P.~Polacik.
\newblock Equilibria with a nontrivial nodal set and the dynamics of parabolic
  equations on symmetric domains.
\newblock {\em J. Differential Equations}, 258(6):1859--1888, 2015.

\bibitem{GPW10}
F.~Gladiali, F.~Pacella, and T.~Weth.
\newblock Symmetry and nonexistence of low {M}orse index solutions in unbounded
  domains.
\newblock {\em J. Math. Pures Appl. (9)}, 93(5):536--558, 2010.

\bibitem{lieberman}
G.M. Lieberman.
\newblock {\em Second order parabolic differential equations}.
\newblock World Scientific Publishing Co., Inc., River Edge, NJ, 1996.

\bibitem{Lu95}
A.~Lunardi.
\newblock {\em Analytic semigroups and optimal regularity in parabolic
  problems}.
\newblock Modern Birkh\"{a}user Classics. Birkh\"{a}user/Springer Basel AG,
  Basel, 1995.
\newblock [2013 reprint of the 1995 original] [MR1329547].

\bibitem{polacik:unbounded}
P.~Polacik.
\newblock Symmetry properties of positive solutions of parabolic equations on
  {$\mathbb R^N$}. {I}. {A}symptotic symmetry for the {C}auchy problem.
\newblock {\em Comm. Partial Differential Equations}, 30(10-12):1567--1593,
  2005.

\bibitem{polacik}
P.~Polacik.
\newblock Estimates of solutions and asymptotic symmetry for parabolic
  equations on bounded domains.
\newblock {\em Arch. Ration. Mech. Anal.}, 183(1):59--91, 2007.

\bibitem{PY04}
P.~Pol\'{a}\v{c}ik and E.~Yanagida.
\newblock Nonstabilizing solutions and grow-up set for a supercritical
  semilinear diffusion equation.
\newblock {\em Differential Integral Equations}, 17(5-6):535--548, 2004.

\bibitem{protter}
M.H. Protter and H.F. Weinberger.
\newblock {\em Maximum principles in differential equations}.
\newblock Springer-Verlag, New York, 1984.
\newblock Corrected reprint of the 1967 original.

\bibitem{QS07}
P.~Quittner and Ph. Souplet.
\newblock {\em Superlinear parabolic problems}.
\newblock Birkh\"{a}user Advanced Texts: Basler Lehrb\"{u}cher. [Birkh\"{a}user
  Advanced Texts: Basel Textbooks]. Birkh\"{a}user Verlag, Basel, 2007.
\newblock Blow-up, global existence and steady states.

\bibitem{torino}
A.~Salda\~na.
\newblock Partial symmetry of solutions to parabolic problems via reflection
  methods.
\newblock {\em Rend. Sem. Mat. Univ. Politec. Torino}, 74(2):105--112, 2016.

\bibitem{saldana:2016}
A.~Salda\~na.
\newblock Qualitative properties of coexistence and semi-trivial limit profiles
  of nonautonomous nonlinear parabolic {D}irichlet systems.
\newblock {\em Nonlinear Anal.}, 130:31--46, 2016.

\bibitem{saldana-weth}
A.~Salda\~na and T.~Weth.
\newblock Asymptotic axial symmetry of solutions of parabolic equations in
  bounded radial domains.
\newblock {\em J. Evol. Equ.}, 12(3):697--712, 2012.

\bibitem{saldana:2015}
A.~Salda\~na and T.~Weth.
\newblock On the asymptotic shape of solutions to {N}eumann problems for
  non-cooperative parabolic systems.
\newblock {\em J. Dynam. Differential Equations}, 27(2):307--332, 2015.

\bibitem{wethsurvey}
T.~Weth.
\newblock Symmetry of solutions to variational problems for nonlinear elliptic
  equations via reflection methods.
\newblock {\em Jahresber. Dtsch. Math.-Ver.}, 112(3):119--158, 2010.

\end{thebibliography}

\end{document}